\documentclass[11pt,oneside]{amsart}
\usepackage{mathtools,amsthm,amssymb,amsfonts,wasysym}
\usepackage{graphicx}
\usepackage{hyperref}
\usepackage{verbatim}
\usepackage[utf8]{inputenc}
\usepackage{csquotes}
\usepackage{geometry}
\usepackage{color}
\usepackage{tikz-cd}
\usetikzlibrary{cd}
\usepackage{caption}
\usepackage{subcaption}
\usepackage{setspace}
\usepackage{todonotes}
\usepackage[shortlabels]{enumitem}
\usepackage[capitalise]{cleveref}
\usepackage{booktabs}
\usepackage{todonotes}
\usepackage{charter,eulervm}
\usepackage[mathscr]{euscript}

\DeclareFontFamily{OT1}{pzc}{}
\DeclareFontShape{OT1}{pzc}{m}{it}{<-> s * [1.10] pzcmi7t}{}
\DeclareMathAlphabet{\mathpzc}{OT1}{pzc}{m}{it}

\theoremstyle{plain}
\newtheorem{theorem}{Theorem}[section]
\newtheorem*{theorem*}{The Contraction Principle}
\newtheorem{corollary}[theorem]{Corollary}
\newtheorem{proposition}[theorem]{Proposition}

\newtheorem{lemma}[theorem]{Lemma}

\newtheorem{fact}[theorem]{Fact}

\theoremstyle{definition}
\newtheorem*{definition*}{Definition}
\newtheorem{definition}[theorem]{Definition}

\newtheorem{example}[theorem]{Example}

\theoremstyle{remark}
\newtheorem{remark}[theorem]{Remark}
\newtheorem*{notation*}{Notation}

\newtheorem{summary}[theorem]{Summary}

\def\s{\sigma}

\def\I{\mathcal{I}}

\def\O{\mathcal{O}}

\def\s{\mathfrak{s}}

\newcommand{\N}{\mathbb{N}}

\newcommand{\xdownarrow}[1]{%
  {\left\downarrow\vbox to #1{}\right.\kern-\nulldelimiterspace}
}

\newcommand{\Frm}{\mathbf{Frm}}

\newcommand{\DLat}{\mathbf{DLat}}
\newcommand{\Pries}{\mathbf{Pries}}

\newcommand{\ignore}[1]{}

\newcommand{\BL}{\mathcal{BL}}
\newcommand{\M}{\mathcal{DM}}

\newcommand{\pH}{\mathpzc{p}\mathcal{H}}

\newcommand{\rb}{\prec}

\newcommand{\meet}{\wedge}
\newcommand{\join}{\vee}

\title[Dedekind-MacNeille and related completions]{Dedekind-MacNeille and related completions: \\
 subfitness, regularity, and Booleanness}
\author{G.~Bezhanishvili}
\address{New Mexico State University}
\email{guram@nmsu.edu}

\author{F.~Dashiell Jr}
\address{CECAT, Chapman University}
\email{dashiell@math.ucla.edu }

\author{M.A.~Moshier}
\address{CECAT, Chapman University}
\email{moshier@chapman.edu}

\author{J.~Walters-Wayland}
\address{CECAT, Chapman University}
\email{walterswayland@chapman.edu}

\date{}

\begin{document}


\begin{abstract}
    Completions play an important r\^ole for studying structure by supplying elements that in some sense ``ought to be." Among these, the Dedekind-MacNeille completion is of particular importance. In 1968 Janowitz provided necessary and sufficient conditions for it to be subfit or Boolean. Another natural separation axiom connected to these is regularity. We explore similar characterizations of when closely related completions are subfit, regular, or Boolean. We are mainly interested in the Bruns-Lakser, ideal, and canonical completions, which (unlike the Dedekind-MacNeille completion) satisfy stronger forms of distributivity. The first two are widely used in pointfree topology, while the latter is of crucial importance in the semantics of modal logic. 
\end{abstract}

\maketitle

\tableofcontents

\section{Introduction}

As is well known, separation properties play an important r\^ole in classifying topological spaces. They have also been investigated in the pointfree setting; see the recent monograph \cite{picado_separation_2021}.   
Among these, subfitness is a particularly important lower separation axiom. It is closely related to $T_1$-separation (a topological space is $T_1$ iff it is $T_D$ and the lattice of open sets is subfit; see, e.g., \cite[p.~24]{picado_separation_2021}). Subfitness originated in the work of Wallman \cite{wallman_lattices_1938} in its dual form under the name of disjunctivity (using the lattice of closed sets of a topological space). It was further studied by B\"uchi \cite{buchi_boolesche_1948} and Pierce \cite{pierce_homorphisms_1954}. In the modern form (using the lattice of open sets) it was explored by Simmons \cite{simmons_lattice_1978} under the name of conjunctivity.   In set theory, disjunctivity is known as separativity and plays an important r\^ole in forcing; see  \cite[pp.199-200]{kunen_set_2011}. 

To distinguish these two notions, we refer to them as {\em $\vee$-subfitness} and {\em $\wedge$-subfitness}, but suppress $\vee$ from the former when no confusion arises. Subfitness is usually defined for frames, that is, complete lattices in which finite meets distribute over arbitrary joins, as these provide a natural generalization of lattices of open sets \cite{picado_frames_2012}. However, they also make perfect sense for arbitrary bounded lattices. We thus define a bounded lattice $A$ to be {\em $\vee$-subfit} if 
\[
\forall a,b~(a\not\le b \Longrightarrow \exists c : a\vee c = 1 \ne b \vee c);
\]
and {\em $\wedge$-subfit} if
\[
\forall a,b~(a\not\le b \Longrightarrow \exists c : a\wedge c \ne 0 = b \wedge c).
\]Our starting point is the following noteworthy observation of
Janowitz \cite{janowitz_section_1968}:

\begin{itemize}
\item A bounded lattice $A$ is $\vee$-subfit (resp.~$\wedge$-subfit) iff its Dedekind-MacNeille completion $\M A$ is $\vee$-subfit (resp.~$\wedge$-subfit).
\item $\M A$ is Boolean iff $A$ is distributive, $\vee$-subfit, and $\wedge$-subfit.
\end{itemize}

While $\M A$ has many desirable properties, it is not always distributive (see, e.g., \cite[Sec.~XII.2]{balbes_distributive_1974}). 
This can be remedied by considering the injective hull (of a $\wedge$-semilattice) constructed by Bruns and Lakser \cite{bruns_injective_1970} (see also Horn and Kimura \cite{horn_category_1971}), which is always a frame.
We call it the Bruns-Lakser completion and denote by 
$\BL A$. 
It is natural to consider the same questions for $\BL A$ as Janowitz considered for $\M A$. This is one of our main aims in this work.
Additionally, we consider two other well-studied completions, the ideal completion $\I A$ and the canonical completion $A\!^\sigma$, and consider the same questions for them. The ideal completion gives the free frame over a distributive lattice, making it an important construction in frame theory \cite{johnstone_stone_1982}, whereas the canonical completion\footnote{ It is more commonly known as the canonical extension, but in this context it is convenient to use a uniform term to describe all four constructions.} plays an important r\^ole in Kripke completeness of modal logics \cite{blackburn_modal_2001}. 

Another well-known separation axiom is regularity. For frames it is situated between subfit and Boolean. We recall that a frame is {\em regular} if
\[
\forall a,b~(a\not\le b \Longrightarrow \exists c : c\rb a \ \& \ c\not\le b),
\]
where $\rb$ is the {\em rather below} relation (see \cref{sec: regularity} for details). Again, this condition makes perfect sense for arbitrary bounded lattices, however distributivity is required to prove that regularity implies subfitness. Because of this, in this paper we concentrate on distributive lattices. The general case will be studied in a forthcoming paper. Thus, the main goal of this paper is to investigate when the four completions $\M A, \BL A, \I A$, and $A\!^\sigma$ of a bounded distributive lattice $A$ are subfit, regular, or Boolean. We will also study when these properties are preserved/reflected by these four completions. 
For this, as expected, counterexamples play an important r\^ole.  
These are often more efficiently described (not to mention discovered) by looking at the Priestley dual. Thus, Priestley duality for distributive lattices, which represents each distributive lattice as the lattice of clopen upsets of a Priestley space (see \cref{sec: prelims} for details), is one of main tools of this work. 

As we will see in \cref{sec: prelims}, we have the poset embeddings
    \[
    A  \hookrightarrow  \M A \hookrightarrow \BL A \hookrightarrow \I A \hookrightarrow A\!^\sigma.
    \]
Thus, $\M A$ is the ``closest" completion to $A$. As we pointed out above, Janowitz 
showed that $\M A$ is subfit iff $A$ is subfit. In addition, we describe 
when subfitness is preserved and/or reflected by a lattice extension, yielding Janowitz's result as a corollary.
On the other hand, since $\M A$ may not be distributive, the issue of regularity for $\M A$ is more complicated (see the previous paragraph).  
We thus assume that $\M A$ is distributive. In this case, we show that 
$\M A$ is regular iff $\BL A$ is regular and $A$ is subfit.  

For a distributive lattice $A$, Janowitz proved that $\M A$ is Boolean iff $A$ is both $\vee$-subit and $\wedge$-subfit. We prove that $\BL A$ is Boolean iff $A$ is $\wedge$-subfit, hence obtaining Janowitz's other result as a corollary. We also characterize when $\BL A$ is subfit and when $\BL A$ is regular. Each turns out to be strictly weaker than the corresponding property of $A$, and we provide necessary and sufficient conditions for when these two conditions are reflected. On the one hand, this relates to the notion of a proHeyting lattice; on the other hand, it results in a stronger notion of regularity for $\BL A$.

ProHeyting lattices were introduced in \cite{ball_dedekind-macneille_2016} to characterize when the Dedekind-MacNeille completion of a distributive lattice is a frame. This plays an important r\^ole in our considerations. For each distributive lattice $A$, we introduce its proHeyting extension $\pH A$ and show that $\BL A$ is exactly the Dedekind-MacNeille completion of $\pH A$. We think of $\pH A$ as a ``finitary" version of $\BL A$, and show that whether $\BL A$ is subfit, regular, or Boolean depends on whether the corresponding properties hold for $\pH A$.

As we will see, the three separation axioms of subfitness, regularity, and Booleanness all collapse for the canonical completion $A\!^\sigma$ and are equivalent to $A$ being Boolean. While for the ideal completion $\I A$, to be Boolean requires the much stronger condition that $A$ is a finite Boolean algebra. As with $A\!^\sigma$, for $\I A$ to be regular is equivalent to $A$ being Boolean, however the ideal completion is more sensitive with respect to subfitness. Indeed, $\I A$ is subfit iff every ideal is an intersection of maximal ideals, a condition that is strictly stronger than subfitness of $A$, but weaker than $A$ being Boolean. 

Utilizing Priestley duality, we characterize these three separation axioms for each of the completions by identifying an appropriate density property. For example, subfitness of $A$ is equivalent to the minimal spectrum of $A$ being dense in its Priestley space $X$, while subfitness of $\I A$ is equivalent to
to the density of the minimal spectrum of $A$ in the Skula topology of the topology of open upsets of $X$ (see \cref{Sec: subfit} for details).

It is natural 
to continue this line of research by considering other separation axioms as well as other completions. From a purely lattice-theoretic perspective, it is also of interest to explore the non-distributive setting,
which we aim to do in the future. 

\section{Preliminaries} \label{sec: prelims}

We start by recalling that a {\em completion} of a poset $P$ is a pair $(C,e)$ where $C$ is a complete lattice and $e:P\to C$ is a poset embedding. As is customary, we identify $P$ with its image $e[P]$ in $C$, and refer to $C$ as a completion of $P$. 

We say that $P$ is {\em join-dense} in $C$ if each element of $C$ is a join from $P$, and define $P$ being {\em meet-dense} in $C$ similarly. Up to isomorphism, each poset $P$ has a unique completion $C$ such that $P$ is both join- and meet-dense in $C$ (see, e.g., \cite[p.~237]{balbes_distributive_1974}). It is known as the {\em Dedekind-MacNeille completion} of $P$ and we denote it by $\M P$. While $\M P$ has many desirable properties (for example, $\M P$ is the injective hull of $P$ in the category $\bf Pos$ of posets and order-preserving maps), it may fail to be distributive. This can be remedied by considering injective hulls in the category $\bf MSLat$ of $\wedge$-semilattices. These were constructed by  Bruns and Lakser \cite{bruns_injective_1970} (see also Horn and Kimura \cite{horn_category_1971}), who showed that the injective hull of each $\wedge$-semilattice $S$ is a frame\footnote{We refer the reader to \cite{johnstone_stone_1982} or \cite{picado_frames_2012} for all basics about frame/locale theory.} in which $S$ is join-dense.  We denote this completion by $\BL S$. 

In this paper we are mainly interested in bounded distributive lattices, so we let $\DLat$ be the category of bounded distributive lattices and bounded lattice homomorphisms. For $A\in\DLat$, in addition to the Dedekind-MacNeille completion $\M A$ and the Bruns-Lakser completion $\BL A$, we consider the {\em ideal completion} $\I A$ (see, e.g., \cite[p.~59]{johnstone_stone_1982}) and the {\em canonical completion} $A\!^\sigma$ (see, e.g., \cite{goldblatt_varieties_1989,gehrke_bounded_1994}).
As follows from the name, $\I A$ is the frame of all ideals of $A$, and the assignment $A\mapsto\I A$ defines a functor from $\DLat$ to the category $\Frm$ of frames (yielding an equivalence between $\DLat$ and the (non-full) subcategory $\bf CohFrm$ of $\Frm$ consisting of coherent frames; see, e.g., \cite[p.~65]{johnstone_stone_1982}). On the other hand, $A\!^\sigma$ is, up to isomorphism, a unique completion of $A$ that is both compact and dense 
\cite[p.~349]{gehrke_bounded_2001}.
Moreover, the assignment $A\mapsto A\!^\sigma$ defines a functor from $\DLat$ to the category $\bf CDLat$ of completely distributive lattices (see, e.g., \cite[Thm.~5.4]{gehrke_bounded_2001} or \cite[Cor.~2.2]{bezhanishvili_canonical_2021}).

We next characterize each of these four completions using Priestley duality for $\DLat$ \cite{priestley_representation_1970,priestley_ordered_1972}. 

\begin{definition}
    A {\em Priestley space} is a compact space $X$ equipped with a partial order $\le$ satisfying the {\em Priestley separation axiom}:
    \[
    \mbox{ if } x\not\le y, \mbox{ then there is a clopen upset } U \mbox{ such that } x\in U \mbox{ and } y \notin U.
    \]
\end{definition}
\begin{remark}
   The Priestley separation axiom implies that the topology has a basis of clopen sets. Thus, the topology of a Priestley space is a Stone topology (compact Hausdorff and zero-dimensional).

\end{remark}

A {\em Priestley morphism} between Priestley spaces is a continuous order-preserving map. Let $\Pries$ be the category of Priestley spaces and Priestley morphisms. 

\begin{theorem} [Priestley duality]
    $\DLat$ is dually equivalent to $\Pries$.
\end{theorem}

\begin{remark} \label{rem: Priestley functors}
    The contravariant functors establishing Priestley duality are constructed as follows. For a Priestley space $X$, let $\sf L(X)$ be the lattice of clopen upsets of $X$, and for a Priestley morphism $f:X\to Y$ let $\sf L(f)=f^{-1}:\sf L(Y)\to\sf L(X)$. 
For a bounded distributive lattice $A$, let $\mathcal S(A)$ be the poset of prime filters of $A$ and let $\mathfrak s_A:A\to\wp(\mathcal S(A))$ be the Stone map given by $\mathfrak s_A(a)=\{x\in \mathcal S(A) : a\in x\}$. We order $\mathcal S(A)$ by inclusion and topologize it by the subbasis
    \[
    \{ \mathfrak s_A(a):a\in A\}\cup\{ \mathfrak s_A(b)^c : b \in A\}.
    \] 
    Then $\mathcal S(A)$ is a Priestley space. Moreover, if $h:A\to B$ is a bounded lattice morphism, then $\mathcal S(h)=h^{-1}:\mathcal S(B)\to\mathcal S(A)$ is a Priestley morpism. These two contravariant functors $\sf L,\mathcal S$ establish Priestley duality.
In particular, $\mathfrak s_A : A\to\sf L\mathcal{S}(A)$ is a lattice isomorphism natural in each coordinate and $\varepsilon_X :X \to \mathcal{S}\sf L(X)$, given by $\varepsilon_X(x)=\{ U\in{\sf L(X)} : x\in U \}$, is an order-homeomorphism natural in each coordinate.
\end{remark}

\begin{remark}
    Priestley duality depends on the Boolean Prime Ideal Theorem (or one of its equivalents), which is a consequence of the Axiom of Choice. Thus, all of our results that depend on Priestley duality are not choice free.
\end{remark}

We will freely use the following well-known facts about Priestley spaces (see, e.g., \cite[pp.~258--259]{davey_introduction_2002}):

\begin{fact} \label{fact: Priestley}
Let $X$ be a Priestley space.
\begin{enumerate}[label=\normalfont(\arabic*), ref = \thefact(\arabic*)]
    \item $\{ U\setminus V : U,V$ clopen upsets$\}$ forms a basis for the topology on $X$. \label[fact]{fact: Priestley 1}
    \item Open upsets (resp.~downsets) are unions of clopen upsets (resp.~downsets). \label[fact]{fact: Priestley 2}
    \item Closed upsets (resp.~downsets) are intersections of clopen upsets (resp.~downsets). \label[fact]{fact: Priestley 3} 
    \item  $\le$ is a closed subset of the product $X^2$. Consequently, ${\uparrow}F$ and ${\downarrow}F$ are closed for each closed subset $F$ of $X$. \label[fact]{fact: Priestley 4}
\end{enumerate}
\end{fact}

Each Priestley space $X$ carries a Stone topology, as well as the topology of open upsets and the topology of open downsets, the former being the join of the latter two. We write $\sf cl$ and $\sf int$ for the closure and interior in the Stone topology, $\sf cl_1$ and $\sf int_1$ for the closure and interior in the topology of open upsets, and $\sf cl_2$ and $\sf int_2$ for the closure and interior in the topology of open downsets. Then
\begin{eqnarray*}
    {\sf cl_1}(S) = {\downarrow}{\sf cl}(S); & \quad & {\sf cl_2}(S)={\uparrow}{\sf cl} (S); \\
    {\sf int_1}(S) = X\setminus{\downarrow}(X\setminus{\sf int}(S)); & \quad & {\sf int_2}(S)=X\setminus{\uparrow}(X\setminus{\sf int}(S)).
\end{eqnarray*}
Therefore,
\begin{eqnarray*}
    x\in{\sf cl_1}(S) \mbox{ iff } {\uparrow}x\cap{\sf cl}(S)\ne\varnothing; & \quad & x\in{\sf cl_2}(S) \mbox{ iff } {\downarrow}x\cap{\sf cl}(S)\ne\varnothing; \\
    x\in{\sf int_1}(S) \mbox{ iff } {\uparrow}x\subseteq{\sf int}(S); & \quad & x\in{\sf int_2}(S) \mbox{ iff } {\downarrow}x\subseteq{\sf int}(S).
\end{eqnarray*}

We are ready to characterize $A\!^\sigma$ and $\I A$ in the language of Priestley spaces. For a Priestley space $X$, let ${\sf Up}(X)$ be the completely distributive lattice 
of upsets and ${\sf OpUp}(X)$ the frame of open upsets of $X$. We then have:

\begin{theorem} \label{lem: char of sigma and I}
    Let $A\in{\DLat}$ and $X$ be the Priestley space of $A$.
    \begin{enumerate}[label=\normalfont(\arabic*), ref = \thetheorem(\arabic*)]
    \item \cite{goldblatt_varieties_1989,gehrke_bounded_1994} $A\!^\sigma \cong  {\sf Up}(X)$. \label[theorem]{lem: char of M and Dinfty-0}
        \item \cite{priestley_representation_1970} $\I A \cong {\sf OpUp}(X)$. \label[theorem]{lem: char of M and Dinfty-1}
    \end{enumerate}
\end{theorem}

\begin{samepage}
    \begin{remark} \label{rem: ideals = open upsets}
    \hfill
    \begin{enumerate}[label=\normalfont(\arabic*), ref = \theremark(\arabic*)]
    \item Alternatively, $A\!^\sigma$ 
   can be described as follows. Let $\mathcal F(A)$ be the frame of filters of $A$. Then there is a Galois connection $(\phi,\psi)$ between the powersets of $\mathcal F(A)$ and $\I(A)$ given by 
    \[
    \phi(Y)=\{ I\in\I(A) : I\cap F\ne\varnothing~\forall F\in Y\}
    \]
    and 
    \[
    \psi(Z)=\{ F\in\mathcal F(A) : F\cap I\ne\varnothing~\forall I\in Z\},
    \]
    and $A\!^\sigma$ is isomorphic to the complete lattice $\mathcal G(\phi,\psi)$ of Galois closed elements (see \cite[Prop.~2.6]{gehrke_bounded_2001}). 

   \item  The isomorphism of \cref{lem: char of M and Dinfty-1} sends $I \in \I A$ to the open upset 
   \[
   \s(I)=\bigcup\{\s(a) : a \in I\}.
   \]
   Therefore, $\s(\bigvee_\alpha I_\alpha) = \bigcup_\alpha \s(I_\alpha)$ and $\s(I \cap J)= \s(I) \cap \s (J)$. 
   This will be used in subsequent sections. \label[remark]{rem: ideals = open upsets 2}
    \end{enumerate}
\end{remark}
\end{samepage}

To characterize $\M A$ and $\BL A$ in terms of the Priestley dual $X$ of $A$, we first recall that $\M A$ can be constructed as the complete lattice of normal ideals of $A$  and $\BL A$ as the frame of D-ideals of $A$ (see the beginning of \cref{sec: pH} for details). Since ideals of $A$ dually correspond to open upsets of $X$, it is sufficient to recognize which open upsets correspond to normal ideals and which to D-ideals.

Consider the following two closure operators on ${\sf OpUp}(X)$:  ${\sf int_1}{\sf cl_2}$ and ${\sf int_1}{\sf cl}$. Observe that the second one is even a nucleus on ${\sf OpUp}(X)$ (see \cite[Prop.~6.1]{bezhanishvili_proximity_2014}). 

\begin{definition}\label{defn MX and BLX}
Let $X$ be a Priestley space.
\begin{enumerate}[label=\normalfont(\arabic*), ref = \thedefinition(\arabic*)]
    \item We call a fixpoint of ${\sf int_1}{\sf cl_2}$ a {\em Dedekind-MacNeille upset} or {\em DM-upset}, and let ${\sf DM}(X)$ be the poset of DM-upsets of $X$.\label[definition]{defn MX and BLX 1}
    \item We call a fixpoint of ${\sf int_1}{\sf cl}$ a {\em Bruns-Lakser upset} or {\em BL-upset},  and let ${\sf BL}(X)$ be the poset of BL-upsets of $X$.\label[definition]{defn MX and BLX 2}
\end{enumerate}
\end{definition}
 
It turns out that normal ideals exactly correspond to DM-upsets, while D-ideals to BL-upsets, and hence we obtain:

\begin{theorem} \label{lem: char of M and Dinfty}
    Let $A\in{\DLat}$ and $X$ be the Priestley space of $A$.
    \begin{enumerate}[label=\normalfont(\arabic*), ref = \thetheorem(\arabic*)]
    
        \item \cite[Thm.~3.5]{bezhanishvili_macneille_2004} $\M A \cong {\sf DM}(X)$.
        \label[theorem]{lem: char of M and Dinfty-2}
        \item \cite[Prop.~6.1, 8.12]{bezhanishvili_proximity_2014} $\BL A \cong {\sf BL}(X)$.
        \label[theorem]{lem: char of M and Dinfty-3}
    \end{enumerate}
\end{theorem}

\begin{remark} \label{fact: Priestley 5}
    We have the following poset inclusions
    \[
    {\sf L}(X)\subseteq {\sf DM}(X) \subseteq {\sf BL}(X) \subseteq {\sf OpUp}(X) \subseteq {\sf Up}(X),
    \]
    most of which are obvious. We only verify that ${\sf DM}(X) \subseteq {\sf BL}(X)$: if $U\in{\sf DM}(X)$ then $U = {\sf int_1cl_2}(U)$, so
    \[
    U\subseteq {\sf int_1cl}(U) \subseteq {\sf int_1} {\uparrow}\,{\sf cl}(U) = {\sf int_1cl_2}(U) = U,
    \]
    and hence $U \in {\sf BL}(X)$. 
    Moreover, ${\sf L}(X)$ is a bounded sublattice of each of the four, ${\sf DM}(X)$ is a bounded sub-$\cap$-semilattice of ${\sf BL}(X)$, ${\sf BL}(X)$ is a sublocale\footnote{We recall that a {\em sublocale} of a frame $L$ is a subset $S\subseteq L$ closed under arbitrary meets and containing relative pseudocomplements $a\to s = \bigvee \{ x \in L : a \wedge x \le s \}$ for all $s\in S$ and $a\in L$. Sublocales are in one-to-one correspondence with nuclei and play an important r\^ole in pointfree topology. We refer to \cite{picado_frames_2012} for details.} of ${\sf OpUp}(X)$, and ${\sf OpUp}(X)$ is a subframe of ${\sf Up}(X)$. That ${\sf BL}(X)$ is a sublocale of ${\sf OpUp}(X)$ can be seen by observing that ${\sf BL}(X)$ is the fixpoints of the nucleus ${\sf int_1cl}$ on ${\sf OpUp}(X)$.   
    For the same reason, 
    joins in ${\sf BL}(X)$ are calculated by the formula
    $
    \bigvee_{{\sf BL}(X)} S = {\sf int_1cl}\left(\bigcup S\right)
    $
    for each $S\subseteq{\sf OpUp}(X)$. This will be used frequently in subsequent sections. 
\end{remark}

\begin{remark} \label{part 2 of above remark}
    As a consequence of \cref{fact: Priestley 5}, we obtain the following poset embeddings
    \[
    A  \hookrightarrow  \M A \hookrightarrow \BL A \hookrightarrow \I A \hookrightarrow A\!^\sigma
    \]
    (which can also be proved directly without the use of Priestley duality). 
    Furthermore, $A$ is isomorphic to a bounded sublattice of each of the four, $\M A$ is isomorphic to a bounded sub-$\meet$-semilattice of $\BL A$, $\BL A$ is isomorphic to a sublocale of $\I A$, and $\I A$ is isomorphic to a subframe of $A\!^\sigma$. For a direct proof that $\BL A$ is isomorphic to a sublocale of $\I A$ we refer to \cite[Thm.~5.17]{bezhanishvili_semilattice_2024}. 
\end{remark}
\cref{tab:my_label} 
summarizes the completions of a distributive lattice that will be considered in this paper, together with their representations. The third column provides the ideal representations and the last column the Priestley representations.

\begin{table}[h]
    \centering
    \begin{tabular}{|c|c|c|c|c|}
    \hline
   Distributive   & $A$ & principal ideals&${\sf L}(X)$ & clopen upsets of $X$ \\
      lattice& &  &&\\
        \hline
   Dedekind-MacNeille& $\M A$ & normal ideals & ${\sf DM}(X)$& DM-upsets of $X$ \\
      completion & &  & & \\ 
         \hline
          Bruns-Lakser &  $\BL A$& D-ideals& $\sf {BL}(X)$&BL-upsets of $X$   \\
           completion& &  &&\\
          \hline
           Ideal &$\I A$   & ideals& $\sf {OpUp}(X)$ &open upsets of $X$ \\
            completion& &  &&\\
        \hline
       Canonical&$A\!^\sigma$ & Galois closed sets &$\sf {Up}(X)$&upsets of $X$\\
  completion& & of ideals &&\\
       \hline
     \end{tabular}
    \caption{Summary of completions and their representations.}
    \label{tab:my_label}
\end{table}

\section{The proHeyting extension}\label{sec: pH}
The concept of a proHeyting lattice was introduced in \cite{ball_dedekind-macneille_2016} to characterize when the Dedekind-MacNeille completion of a distributive lattice is a frame. We show that with each $A \in \DLat$ we can associate a natural proHeyting lattice which we call the proHeyting extension of $A$. This extension will enable us to characterize when $\BL A$ is subfit, regular, or Boolean. 

As we pointed out in the previous section, for $A\in\DLat$, 
the Dedekind-MacNeille completion  $\M A$ can be constructed as the complete lattice of normal ideals (see, e.g., \cite[p.~63]{gratzer_lattice_2011}), while the Bruns-Lakser completion $\BL A$ as the frame of D-ideals of $A$ (\cite{bruns_injective_1970}). 
To recall relevant definitions, for $S\subseteq A$ let $S^u$ denote the set of upper bounds and $S^\ell$ the set of lower bounds; $S$ is \emph{admissible} if 
    $\bigvee S$ exists and $a\wedge\bigvee S = \bigvee\{ a\wedge s : s \in S\}$ for each $a\in A$.

\begin{definition}\label{defn: normal and D-ideal}
Let $A\in\DLat$.
\begin{enumerate}
\item An ideal $N$ of $A$ is {\em normal} if $N = N^{u\ell}$ (equivalently, $N$ is an intersection of principal ideals).
      \item A downset $D$ of $A$ is a {\em D-ideal} if $\bigvee S\in D$ for each admissible $S \subseteq D$. 
    \end{enumerate}    
\end{definition}
    The definition of a proHeyting lattice requires the notion of a relative annihilator:
   \begin{definition}\cite{mandelker_relative_1970}
 A {\em relative annihilator} of a lattice $A$ is a downset of the form 
\[
\langle a,b \rangle = \{ x\in A : a\wedge x\le b\}.
\]
\end{definition}

\begin{remark}\label{rem: rel ann}\
    \begin{enumerate}[label=\upshape(\arabic*), ref = \thetheorem(\arabic*)]
        \item For bounded $A$, ${\downarrow} a = \langle 1,a \rangle$ for all $a \in A$

        \item If $A$ is distributive then $\langle a,b \rangle$ is an ideal. In fact, $\langle a,b \rangle$ is the relative pseudocomplement ${\downarrow} a \to {\downarrow} b$ in $\I A$, and $\langle a,b \rangle$ is principal iff $a\to b$ exists in $A$.\label[remark]{rem: rel ann 2} 
  
      \end{enumerate}
\end{remark}
 Relative annihilators may be thought of as the building blocks for $\BL A$:\footnote{The referee pointed out to us that the dual statement of \cref{lem: ann is D} for filters is given in \cite[Prop.~5.5]{jakl_canonical_2024}.}

\begin{lemma} \label{lem: ann is D} 
\hfill
\begin{enumerate}[label=\upshape(\arabic*), ref = \thetheorem(\arabic*)]
    \item Each relative annihilator is a D-ideal.\label[lemma]{lem: ann is D 1} 
    \item Each D-ideal is an intersection of relative annihilators.\label[lemma]{lem: ann is D 2}   
\end{enumerate}    
\end{lemma}
\begin{proof}
    (1) Let $S\subseteq \langle a,b \rangle$ be admissible. Then $a\wedge\bigvee S = \bigvee \{ a \wedge s : s \in S\} \le b$ since $a\wedge s\le b$ for each $s\in S$. Thus, $\bigvee S\in \langle a,b \rangle$, and hence $\langle a,b \rangle$ is a D-ideal. 

    (2) This follows from \cite[Lem.~3.5,~3.12]{bezhanishvili_semilattice_2024}.
  \end{proof}

\begin{definition} \cite[Def.~3.1]{ball_dedekind-macneille_2016} 
     $A\in{\DLat}$  is {\em proHeyting} if $\langle a,b \rangle$ is a normal ideal for each $a,b\in A$.\footnote{In \cite{ball_dedekind-macneille_2016}, the authors write $b{\downarrow}a$ instead of $\langle a,b \rangle$.}  
\end{definition}

\begin{remark}
    By \cref{rem: rel ann 2}, every Heyting lattice is proHeyting. 
    Therefore, proHeyting lattices provide a natural generalization of Heyting lattices. That this is indeed a proper generalization is shown in
    \cite[Sec.~3]{ball_dedekind-macneille_2016}. 
\end{remark}

 We point out that the importance of proHeyting is in the following:

\begin{proposition} \label{thm: MA frame}
   \cite[Thm.~6.3]{ball_dedekind-macneille_2016} \label[proposition]{thm: MA frame 1}
    For $A\in{\DLat}$, the following are equivalent:
    \begin{enumerate}[label=\upshape(\arabic*), ref = \theproposition(\arabic*)]
        \item $\M A$ is a frame.
        \item $\M A \cong \BL A$.
        \item $A$ is proHeyting.
    \end{enumerate}
    
\end{proposition}

\begin{remark}
An alternate proof of the above result is given in 
\cite[ Thm.~3.15]{bezhanishvili_semilattice_2024}.
 \end{remark}
  We next introduce our proHeyting extension of $A\in\DLat$. Viewing $\BL A$ as the frame of D-ideals of $A$, it follows from \cref{lem: ann is D} that relative annihilator ideals of $A$ form a meet-dense subposet of $\BL A$.

\begin{definition}
    For $A\in\DLat$, the \emph{proHeyting extension} of $A$ is the bounded sublattice $\pH A$ of $\BL A$ generated by the relative annihilator ideals of $A$. 
\end{definition}

\begin{remark} \label{rem: RA and pHA}
As we will see in \cref{cor: pH A a}, $\pH A$ is always a proHeyting lattice, justifying the name. The following is a convenient way to think about $\pH A$: Let $\mathcal{RA}(A)$ be the poset of finite intersections of relative annihilator ideals. Then $\mathcal {RA}(A)$ is a bounded sub-$\wedge$-semilattice of $\BL A$, and $\pH A$ consists of the finite joins in $\BL A$ of elements of $\mathcal {RA}(A)$. In fact, as is shown in \cite[Sec.~4]{bezhanishvili_subfitness_2024}, $\pH A$ is exactly the distributive lattice envelope of $\mathcal {RA}(A)$.
\end{remark}
 
Since $a\mapsto{\downarrow}a$ embeds $A$ into $\pH A$, from now on we will identify $A$ with its image in $\pH A$. 
Thus, we have the following lattice inclusions
$A \subseteq \pH A \subseteq \BL A$. In the next proposition we characterize when these inclusions are equalities.

\begin{proposition} Let $A\in\DLat$. 
    \begin{enumerate}[label=\normalfont(\arabic*), ref = \theproposition(\arabic*)]
        \item  $A = \pH A$ iff $A$ is Heyting.
        \item $\pH A =\BL A$ iff $\pH A$ is a frame.
    \end{enumerate}
\end{proposition}

\begin{proof}
    (1) $A$ is Heyting iff $a\to b$ exists in $A$ for each $a,b\in A$, which is equivalent to $\langle a,b \rangle$ being principal for each $a,b\in A$. The latter means that $A = \pH A$. 

    (2) Clearly if $\pH A =\BL A$ then $\pH A$ is a frame. For the converse, since $A$ is join-dense in $\BL A$, for each $x\in \pH A$ we have $x=\bigvee_{\BL A} S$ for some $S\subseteq A$. Since $x\in \pH A$, $x=\bigvee_{\pH A} S$, and so $A$ is join-dense in $\pH A$. By \cite[Thm.~7.4]{ball_dedekind-macneille_2016}, each frame having $A$ as a base contains $\BL A$. Thus, if $\pH A$ is a frame then $\pH A =\BL A$. 
\end{proof}
The following shows that the Bruns-Lakser completion of a distributive lattice can always be realized as the Dedekind-MacNeille completion of a suitably enlarged lattice (the proHeyting extension).
\begin{theorem} \label{prop: M of pH A}
    For any $A \in \DLat$, $\M (\pH A) \cong \BL A$.
\end{theorem}

\begin{proof}
  Since $A$ is join-dense in $\BL A$, we have that $\pH A$ is also join-dense in $\BL A$. By \cref{lem: ann is D 2}, the poset of relative annihilators is meet-dense in $\BL A$, and hence so is $\pH A$. Because $\pH A$ is both join- and meet-dense in $\BL A$, we conclude that $\BL A$ is isomorphic to $\M (\pH A)$ (see, e.g., \cite[p.~237]{balbes_distributive_1974}).
\end{proof}

\begin{corollary}\ \label{cor: pH A} 
\begin{enumerate}[label=\normalfont(\arabic*), ref = \thecorollary(\arabic*)]
\item $\pH A$ is proHeyting.\footnote{
 Note that it is not necessarily Heyting (as will be demonstrated in a follow-up paper). 
} \label[corollary]{cor: pH A a}
\item $\BL (\pH A) \cong \BL A$. \label[corollary]{cor: pH A b}
\end{enumerate}
\end{corollary}

\begin{proof}
    (1) By \cref{prop: M of pH A}, $\M (\pH A)$ is a frame, so $\pH A$ is proHeyting by \cref{thm: MA frame 1}.

    (2) By \cref{prop: M of pH A}, $\BL A \cong \M(\pH A)$. On the other hand, 
    by (1) and \cref{thm: MA frame 1}, $\M(\pH A) \cong \BL (\pH A)$. The result follows. 
\end{proof}

As we saw in \cref{lem: char of M and Dinfty-3}, $\BL A$ is isomorphic to ${\sf BL}(X)$. We now identify the bounded sublattice of ${\sf BL}(X)$ that is isomorphic to $\pH A$.

\begin{lemma}\label{lem: annih char}
    Let $A\in\DLat$, $X$ be the Priestley space of $A$, and $I$ an ideal of $A$. 
    \begin{enumerate}[label=\normalfont(\arabic*), ref = \thelemma(\arabic*)]
        \item $I$ is a relative annihilator iff
    $
    \s(I)=X\setminus {\downarrow}(\s(a)\setminus \s(b))
    $ for some $a,b \in A$. \label[lemma]{lem: annih char 1}
    \item  $I \in \mathcal{RA} (A)$ iff
    $
    \s(I) = X\setminus {\downarrow} K
    $
    for some clopen $K\subseteq X$. \label[lemma]{lem: annih char 2}
    \end{enumerate}
\end{lemma}

\begin{proof}
    (1) For $a,b \in A$, we have 
    \begin{eqnarray*}
        c \in \langle a,b \rangle & \mbox{iff} & a\wedge c \le b \\
        & \mbox{iff} & \s(a)\cap\s(c)\subseteq\s(b) \\
        & \mbox{iff} & 
         \left(\s(a)\setminus\s(b)\right)\cap\s(c)=\varnothing \\ 
        & \mbox{iff} & {\downarrow}(\s(a)\setminus\s(b))\cap\s(c)=\varnothing \ \text{  (since $\s(c)$ is an upset)}  \\ 
        & \mbox{iff} & \s(c)\subseteq X\setminus {\downarrow}(\s(a)\setminus \s(b)).
    \end{eqnarray*}
    Therefore, by \cref{rem: ideals = open upsets 2}, $\s(\langle a,b \rangle)=X\setminus {\downarrow}(\s(a)\setminus \s(b))$. The result follows since, by definition, $I$ is a relative annihilator iff $I = \langle a,b \rangle$ for some $a,b \in I$. 

    (2) Let $I \in \mathcal{RA}(A)$. Then $I = \bigcap_{i=1}^n \langle a_i,b_i \rangle$ for some $a_i,b_i \in A$. Therefore, by \cref{rem: ideals = open upsets 2} and (1) we have
    \begin{eqnarray*}
    \s(I) & = & \s\left(\bigcap_{i=1}^n \langle a_i,b_i \rangle\right) = \bigcap_{i=1}^n \s\left(\langle a_i,b_i \rangle\right) \\
    & = & \bigcap_{i=1}^n \left(X\setminus {\downarrow}(\s(a_i)\setminus \s(b_i))\right) = X \setminus \bigcup_{i=1}^n {\downarrow}(\s(a_i)\setminus \s(b_i)) \\
    & = & X \setminus {\xdownarrow{.35cm}} \bigcup_{i=1}^n (\s(a_i)\setminus \s(b_i)). 
    \end{eqnarray*}
    Because each clopen of $X$ is of the form $K=\bigcup _{i=1}^n (\s(a_i)\setminus \s(b_i))$ for some $a_i,b_i \in A$, the result follows.
\end{proof}

Let $X$ be a Priestley space and $K\in{\sf Clop}(X)$. Then $X\setminus{\downarrow}K\in{\sf BL}(X)$, and we define ${\sf pH}(X)$ to be the finite joins in ${\sf BL}(X)$ of elements of this form. We then have:

\begin{proposition} \label{prop: pH A = pH X}
$\pH A$ is isomorphic to ${\sf pH}(X)$.
\end{proposition}

\begin{proof}
    By \cref{lem: char of M and Dinfty-3}, $\BL A$ is isomorphic to ${\sf BL}(X)$. By \cref{rem: RA and pHA}, $\pH A$ is exactly the finite joins of $\mathcal{RA} (A)$ in $\BL A$, which is a sub-$\wedge$-semilattice of $\BL A$. By definition, ${\sf pH}(X)$ is  the finite joins in ${\sf BL}(X)$ of elements of the form $X\setminus{\downarrow}K$, where $K\in{\sf Clop}(X)$. The latter sub-$\wedge$-semilattice of ${\sf BL}(X)$ is isomorphic to $\mathcal{RA} (A)$ by \cref{lem: annih char 2}. The result follows. 
\end{proof}

 We point out that $\pH A$ was constructed as the sublattice of $\BL A$ generated by the relative annihilators of $A$. This may be generalized as follows: by \cite[Thm.~5.17]{bezhanishvili_semilattice_2024}, up to isomorphism, the frames that have $A$ as a lattice base form the interval $[\BL A,\I A]$ in the co-frame of sublocales of $\I A$. For each such $L$, we may consider the sublattice of $L$ generated by the relative annihilators of $A$. In our opinion, these constructions deserve further investigation.

\section{Subfitness}\label{Sec: subfit}
In this section we characterize when the various completions considered in this paper are subfit. This yields necessary and sufficient conditions for when subftness is preserved or reflected by these completions. 

We begin by recalling a characterization of subfitness of a distributive lattice in terms of its Priestley dual.
For $A\in{\DLat}$ and $X$ the Priestley space of $A$, we denote by $\min X$ the set of minimal points of $X$. Since for each $x\in X$ there exists $m\in\min X$ such that $m\le x$,
we see that $\min X$ is order-dense in $X$ in that $X={\uparrow}\min X$ (see, e.g, \cite[p.~42]{priestley_ordered_1984}). The lattice $A$ being subfit exactly amounts to $\min X$ being topologically dense in $X$: 

\begin{proposition} \label{prop: char of subfit}
For $A\in{\DLat}$ and its Priestley space $X$, the following are equivalent:
    \begin{enumerate}[label=\normalfont(\arabic*), ref = \theproposition(\arabic*)]
        \item $A$ is subfit.
        \item $a\not\le b$ implies there is a maximal ideal $M$ with $b \in M$ and $a \notin M$. \label[proposition]{prop: char of subfit 2}
        \item Every principal ideal is an intersection of maximal ideals. \label[proposition]{prop: char of subfit 2a}
        \item $\min X$ is dense in $X$. \label[proposition]{prop: char of subfit 3}
    \end{enumerate}
\end{proposition}

\begin{proof}
    (1)$\Leftrightarrow$(2) is proved in \cite[Prop.~3.2]{delzell_conjunctive_2021}, (2)$\Leftrightarrow$(3) is obvious, and (1)$\Leftrightarrow$(4) is proved in \cite[Lem.~3.1]{bezhanishvili_new_2023} when $A$ is a frame, but the same proof works for an arbitrary $A\in\DLat$ (see  also \cite[Prop.~5.2]{bezhanishvili_subfitness_2024}). 
\end{proof}

We next consider subfitness of the Dedekind-MacNeille completion. As pointed out in the introduction, this was first studied in \cite{janowitz_section_1968}. Janowitz does not assume distributivity and uses a formally weaker form of subfitness which, as  will be shown in a forthcoming paper, is equivalent to subfitness as defined above. 
The next result is slightly more general in that completeness of the lattice $B$ plays no r\^ole.
\begin{theorem} \label{thm: Janowitz}
Let $A$ be a bounded sublattice of $B$. 
\begin{enumerate}[label=\normalfont(\arabic*), ref = \thetheorem(\arabic*)]
\item If $A$ is meet-dense in $B$ and $B$ is subfit then $A$ is subfit.
\item If $A$ is both join- and meet-dense in $B$ and $A$ is subfit then $B$ is subfit.
\end{enumerate}
\end{theorem}

\begin{proof} 
(1) Suppose $a\nleq b$ in $A$. Since $B$ is subfit, there is $x\in B$ with $a\join x=1$ but $b\join x\ne 1$. Because $A$ is meet-dense in $B$, there is $c\in A$ with $b\join x\leq c < 1$. Therefore, $a\join c = 1$ but $b\join c \ne 1$ in $A$, and hence $A$ is subfit. 

(2) Suppose $x\nleq y$ in $B$. Since $A$ is both join- and meet-dense in $B$, there are $a,b\in A$ with $a\le x$, $y\le b$, and $a\nleq b$. Because $A$ is subfit, there is $c\in A$ with $a\vee c = 1$ but $b\vee c\ne 1$ in $A$. Therefore, $x\vee c=1$ but $y\vee c\neq 1$ in $B$. Thus, $B$ is subfit.
\end{proof}

\begin{corollary}\cite[Cor. 2.5]{janowitz_section_1968}
     \label{thm: MA subfit implies A subfit}
    Let $A$ be a bounded lattice. Then $\M A$ is subfit iff $A$ is subfit.
\end{corollary}

\begin{proof} 
Identify $A$ with a sublattice of $\M A$ and observe that $A$ is both join- and meet-dense in $\M A$. Now apply \cref{thm: Janowitz}.
\end{proof}

We now consider subfitness for the Bruns-Lakser completion of $A$. As we saw in \cref{thm: MA subfit implies A subfit}, subfitness of $\M A$ is equivalent to that of $A$; in comparison,  subfitness of $\BL A$ is equivalent to that of $\pH A$: 

\begin{theorem}\label{thm: main} 
 For $A\in\DLat$, $\BL A$ is subfit iff $\pH A$ is subfit.
\end{theorem}
\begin{proof}
By \cref{prop: M of pH A}, $\BL A \cong \M(\pH A)$. Therefore, by \cref{thm: MA subfit implies A subfit}, $\BL A$ is subfit iff $\pH A$ is subfit.
\end{proof}
\begin{remark} \label{rem: min pX}
For $A\in\DLat$, let $X$ be the Priestley dual of $A$ and $pX$ the Priestley dual of $\pH A$. 
As an immediate consequence of \cref{prop: char of subfit,thm: main}, we obtain: $\BL A$ is subfit iff $\min pX$ is dense in $pX$. It is interesting to give an explicit description of $pX$, which we aim to do in future work.
\end{remark}
 As with $\M A$, we show that $\BL A$ preserves subfitness. However, unlike $\M A$, it does not reflect it. For this we need the following:
\begin{lemma} \label{prop: subfit implies proH}
   \cite[Lem.~7.4]{ball_lindelof_2017}  If $A\in{\DLat}$ 
    is subfit then $A$ is proHeyting. 
   \end{lemma} 
 
\begin{samepage}
\begin{theorem} \label{thm: Fred Q1}
Let $A\in{\DLat}$. 
\begin{enumerate}[label=\normalfont(\arabic*), ref = \thetheorem(\arabic*)]
    \item $A$ subfit implies that $\BL A$ is subfit. \label[theorem]{thm: Fred Q1 a}
    \item  $\BL A$ subfit does not imply that $A$ is subfit. \label[theorem]{thm: Fred Q1 b}
\end{enumerate}    
\end{theorem}
\end{samepage}

\begin{proof} 
(1) If $A$ is subfit, then $A$ is proHeyting by \cref{prop: subfit implies proH}, hence $\M A \cong \BL A$ by \cref{thm: MA frame}. Therefore, $\BL A$ is subfit by \cref{thm: MA subfit implies A subfit}.

(2) Let $X$ be the space depicted in \cref{Figure 1}, where $\{x_0,x_1,\dots\}\cup\{y\}$ is the set of isolated points and $x_\infty$ is the limit of $\{x_0,x_1,\dots\}$. In other words, $X$ is the one-point compactification of the discrete space $\{x_0,x_1,\dots\}\cup\{y\}$. Define the partial order on $X$ as shown in the figure (that is, $x_\infty < y$ is the only nontrivial relation). 

\begin{center}
    \begin{figure}[ht]
\begin{tikzpicture}[scale=.75]
 \node  (0) at (0,1) {$\bullet$};
         \node  (0_) at (0,0.5) {$\scriptstyle{x_0}$}; 
\node  (1) at (1,1) {$\bullet$};\node  (1_) at (1,0.5) {$\scriptstyle{x_1}$}; 
\node  (2) at (2,1) {$\bullet$};\node  (2_) at (2,0.5) {$\scriptstyle{x_2}$}; 
\node  (inf) at (7,1) {$\bullet$};\node  (inf_) at (7,0.5) {$\scriptstyle{x_\infty}$};
\node (X) at (7,3) {$\bullet$};\node (X_) at (7.5,3) {$\scriptstyle{y}$};
\draw[dotted] (2)-- (inf);
\draw  (inf)-- (X);
\end{tikzpicture}
\caption{A Priestley space $X$ with ${\sf BL} (X)$ subfit but ${\sf L} (X)$ not.}
        \label{Figure 1}
    \end{figure}
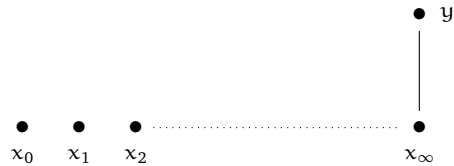
\end{center}   

It is straightforward to check that $X$ equipped with this partial order is a Priestley space. Let $A={\sf L}(X)$. By Priestley duality, we identify $X$ with the Priestley space of $A$. Since $\min X=\{x_0,x_1,\dots\}\cup\{x_\infty\}$ is not dense in $X$, by \cref{prop: char of subfit} $A$ is not subfit. 
We show that $\BL A$ is subfit.
By \cref{lem: char of M and Dinfty}(2), we identify $\BL A$ with ${\sf BL}(X)$.
Let $U,V\in {\sf BL}(X)$ with $U\not\subseteq V$. We have three cases to consider:

First suppose there is $x_n \in U\setminus V$. Then $W:=X\setminus\{x_n\}$ is a clopen upset, hence an element of ${\sf BL}(X)$, such that $U\cup W=X$ and $V\cup W\ne X$. Since the join in ${\sf BL}(X)$ is given by $D\vee E={\sf int_1}{\sf cl}(D\cup E)$ (see \cref{fact: Priestley 5}) and $x_n$ is isolated, we obtain $U\vee W=X$ and $V\vee W\ne X$. 

Next suppose $y \in U \setminus V$. Let $W=\{x_0,x_1,\dots\}$. 
Then $W\in {\sf BL}(X)$. Because $y\in U$, we have $U\cup W$ is dense in $X$, so $U\vee W=X$. On the other hand, $y\notin V$ implies $x_\infty\notin V$, so $V\subseteq W$, and hence $V\vee W=W\ne X$. 

Finally, let $x_\infty \in U\setminus V$. Then $y\in U$, so if $y\notin V$, we are in the previous case. Suppose $y\in V$. If $V$ contains a cofinite subset of $\{x_0,x_1,\dots\}$, then $x_\infty\in{\sf int_1}{\sf cl}(V)=V$, a contradiction. So $V\cap\{x_0,x_1,\dots\}$ is finite. On the other hand, $U$ must contain a cofinite subset of $\{x_0,x_1,\dots\}$ because, by \cref{fact: Priestley 2}, there is $W\in{\sf L}(X)$ such that $x_\infty \in W \subseteq U$. 
Since $W$ is clopen containing $x_\infty$, which is the limit of $\{ x_0,x_1,\dots\}$, $W$ must contain a cofinite subset of $\{ x_0,x_1,\dots\}$. 
Thus, there must exist $x_n \in U \setminus V$, and we are in the first case, completing the proof that ${\sf BL}(X)$ is subfit.
\end{proof}
Putting \cref{thm: main,thm: Fred Q1} together yields:
\begin{corollary} If $A \in \DLat$ is subfit, then so is $\pH A$, but not conversely.
\end{corollary}

To obtain an analogue of \cref{thm: MA subfit implies A subfit} for $\BL A$ (equivalently for $\pH A$), we need the additional requirement that $A$ is proHeyting:

 \begin{proposition}
    Let $A\in\DLat$ be proHeyting. Then
      $\BL A$ is subfit iff $A$ is subfit.        
\end{proposition}

\begin{proof}
    The reverse implication follows from \cref{thm: Fred Q1 a}.
  For the forward implication, since $A$ is proHeyting, \cref{thm: MA frame 1} implies that $\BL A\cong\M A$, thus it is enough to apply \cref{thm: MA subfit implies A subfit}.
\end{proof}

We next consider subfitness for the ideal completion.
As we pointed out in \cref{rem: ideals = open upsets 2}, \[I\mapsto\s(I) = \bigcup\{ \s(a) : a \in I \}\] is an isomorphism of $\I(A)$ and ${\sf OpUp}(X)$, where $X$ is the Priestley space of $A$. Similarly, \[F\mapsto\s(F) = \bigcap\{\s(a) : a \in F\}\] is an isomorphism between the filters of $A$ and the closed upsets of $X$ (ordered by $\supseteq$). The following lemma belongs to folklore. 

\begin{lemma} \label{lem: filter-ideal}
    Let $A\in\DLat$, $I$ be an ideal, and $F$ a filter of $A$. Then $I\cap F\ne\varnothing$ iff $\s(F)\subseteq\s(I)$.
\end{lemma}

\begin{proof}
    If $a\in F\cap I$, then $\s(F)\subseteq\s(a)\subseteq\s(I)$. Conversely, since $X$ is compact, from $\s(F)\subseteq\s(I)$ it follows that there is $a\in A$ with $\s(F)\subseteq\s(a)\subseteq\s(I)$. Thus, $a\in F\cap I$.
\end{proof}
For the next theorem, we utilize the Skula topology \cite{skula_reflective_1969} --- a convenient tool in the study of topological spaces. We recall that the {\em Skula topology} of a topological space $X$ is the topology generated by the basis 
\[
\{ U \setminus V : U,V \mbox{ open in } X \}.
\]

\begin{theorem} \label{I-subfit}
Let $A\in\DLat$ and $X$ be the Priestley space of $A$. The following are equivalent:
    \begin{enumerate}[label=\upshape(\arabic*), ref = \theproposition(\arabic*)]
        \item $\I A$ is subfit.
         \item For any $a\not \in J \in \I A$, there is a maximal ideal $M$ with $J \subseteq M$ and $a\notin M$. \label[proposition]{I-subfit 2}
        \item Every ideal is an intersection of maximal ideals. \label[proposition]{I-subfit 3}
        \item For any $a\not \in J \in \I A$, there is a minimal prime filter $P$ such that $a \in P$ and $P \cap J = \varnothing$. 
        \item $\min X$ is dense in the Skula topology of ${\sf OpUp}(X)$. 
        \item $x\in{\sf cl}(\min {\downarrow} x)$ for each $x \in X$. \label[proposition]{I-subfit 6}
    \end{enumerate}
\end{theorem} 

\begin{proof}
  $(1) \Rightarrow (4)$: Suppose $a \notin J$, so ${\downarrow} a \not\subseteq J$. Since $\I A$ is subfit, there is $I \in \I A$ such that ${\downarrow} a \vee I = A$ but $J \vee I \ne A$. Therefore, by \cref{rem: ideals = open upsets 2}, $\s(a) \cup \s(I) = X$ but $\s(J) \cup \s(I) \ne X$. The latter equality delivers $P \in \min X$ such that $P \notin \s(J) \cup \s(I)$. Since $P \in \s(a) \cup \s(I)$ and $P \notin \s(I)$, we must have $P \in \s(a)$, so $a\in P$. Also, because $P\notin \s(J)$, \cref{lem: filter-ideal} implies that $P\cap J=\varnothing$. Thus, (4) holds. 

$(2) \Leftrightarrow (3)$: This is obvious.

$(2) \Leftrightarrow (4)$: It is well known that prime ideals are exactly the complements of prime filters (see, e.g., \cite[p.~68]{balbes_distributive_1974}) and, under that correspondence, maximal ideals are the complements of minimal prime filters. Thus, for the forward implication put $P= A \setminus M$, and for the reverse implication put $M=A\setminus P$.

$(4) \Rightarrow (5)$: Let $U,V\in{\sf OpUp}(X)$ and $U\setminus V\ne\varnothing$. By \cref{fact: Priestley 2}, there is a clopen upset $W\subseteq U$ such that $W\not\subseteq V$. Since $W=\s(a)$ and $V=\s(J)$ for some $a \in A$ and $J \in \I A$, we have $a\notin J$. So by (4) there is a minimal prime filter $P$ such that $a\in P$ and $P\cap J=\varnothing$. Therefore, $P\in\min X\cap W$ and $P\notin V$ (by \cref{lem: filter-ideal}). Thus, $\min X\cap(U\setminus V)\ne\varnothing$, and hence $\min X$ is dense in the Skula topology of ${\sf OpUp}(X)$.

$(5) \Rightarrow (1)$: Suppose $I \not\subseteq J$ in $\I A$. Then there is $a\in I\setminus J$. Therefore, $\s(a)\not\subseteq\s(J)$, so $\s(a)\setminus\s(J)\ne\varnothing$. By (5), there is a minimal prime filter $P$ such that $P\in\s(a)\setminus\s(J)$. 
Let $W = X \setminus \{ P \}$. Then $W$ is an open upset, $\s(I) \cup W = X$ (because $P \in \s(I)$), and $\s(J) \cup W \ne X$ (because $P \notin \s(J)$ and $P \notin W$). Let $H = \{ a \in A : \s(a)\subseteq W \}$ be the ideal of $A$ corresponding to $W$. By \cref{rem: ideals = open upsets 2}, $I \vee H = A$ but $J \vee H \ne A$, and hence $\I A$ is subfit.

$(3) \Leftrightarrow (6)$: Since every ideal is an intersection of prime ideals, (3) is equivalent to every prime ideal being the intersection of maximal ideals containing it. Recall (see, e.g., \cite[p.~54]{priestley_ordered_1984} or \cite[p.~385]{bezhanishvili_bitopological_2010}) that an ideal $I$ is prime iff $\s(I)=X\setminus{\downarrow}x$ for some $x\in X$ and $I$ is maximal iff $\s(I)=X\setminus\{x\}$ for some $x\in\min X$. Since the meet in ${\sf OpUp}(X)$ is given by ${\sf int_1}\bigcap$, we see that the above is equivalent to $X\setminus{\downarrow}x = {\sf int_1}\bigcap \{ X\setminus\{y\}: y\in\min X \mbox{ and } y\le x\}$ for each $x\in X$. This, in turn, is equivalent to ${\downarrow}x = {\sf cl_1}(\min{\downarrow}x)$ for each $x\in X$. Finally, since ${\sf cl_1}={\downarrow} \, {\sf cl}$ and ${\sf cl}\min{\downarrow}x \subseteq {\downarrow}x$, the last condition is equivalent to (6). 
\end{proof}

\begin{proposition} \label{prop: subfitness A vs IA}
Let $A\in\DLat$.
\begin{enumerate}[label=\upshape(\arabic*), ref = \theproposition(\arabic*)]
    \item $\I A$ subfit implies $A$ subfit. \label[proposition]{prop: subfitness A vs IA 2}
    \item $A$ subfit does not imply $\I A$ subfit. \label[proposition]{prop: subfitness A vs IA 1}
   \end{enumerate}
\end{proposition}

\begin{proof}
(1) This is immediate from \cref{I-subfit 3,prop: char of subfit 2a}. 

(2) Let $X$ be the space shown in \cref{Fig 2}, where $\{x_0,x_1,\dots\}\cup\{y_0,y_1,\dots\}$ is the set of isolated points of $X$, $x_\infty$ is the limit of $\{x_0,x_1,\dots\}$ and $y_\infty$ is the limit of $\{y_0,y_1,\dots\}$ (that is, $X$ is the two-point compactification of the discrete space $\{x_0,x_1,\dots\}\cup\{y_0,y_1,\dots\}$). Clearly $X$ is a Stone space. Let the partial order on $X$ be as depicted in \cref{Fig 2} (that is, the only nontrivial relation is $x_\infty < y_\infty$). 

     \begin{center} 
    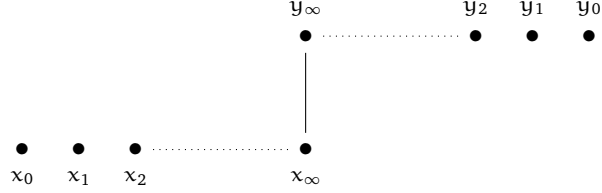
\begin{figure}[hbt!]
\begin{tikzpicture}[scale=.75]
 \node  (x0) at (0,1) {$\bullet$};
         \node  (x0_) at (0,0.5) {$\scriptstyle{x_0}$}; 
\node  (x1) at (1,1) {$\bullet$};\node  (x1_) at (1,0.5) {$\scriptstyle{x_1}$}; 
\node  (x2) at (2,1) {$\bullet$};\node  (x2_) at (2,0.5) {$\scriptstyle{x_2}$}; 
\node  (xinf) at (5,1) {$\bullet$};\node  (xinf_) at (5,0.5) {$\scriptstyle{x_\infty}$}; 
\node  (yinf) at (5,3) {$\bullet$};\node  (yinf_) at (5,3.5) {$\scriptstyle{y_\infty}$};
\node  (y2) at (8,3) {$\bullet$};\node  (y2_) at (8,3.5) {$\scriptstyle{y_2}$};
\node  (y1) at (9,3) {$\bullet$};\node  (y1_) at (9,3.5) {$\scriptstyle{y_1}$};
\node  (y0) at (10,3) {$\bullet$};\node  (y0_) at (10,3.5) {$\scriptstyle{y_0}$};
\draw[dotted] (x2) --  (xinf);
\draw[dotted] (yinf) --  (y2);
\draw  (xinf)-- (yinf);
\end{tikzpicture}
\caption{A Priestley space $X$ with ${\sf L}(X)$ subfit but ${\sf OpUp}(X)$ not.}
        \label{Fig 2}
    \end{figure}
\end{center}
It is straightforward to check that $X$ equipped with this partial order is a Priestley space. Let $A = {\sf L}(X)$; as before, identify $X$ with the Priestley space of $A$. It is clear that 
\[\min X=\{x_0,x_1,\dots\}\cup\{x_\infty\}\cup\{y_0,y_1,\dots\}\]
is dense in $X$. Thus, $A$ is subfit by \cref{prop: char of subfit}. On the other hand, $\min({\downarrow}y_\infty)=\{x_\infty\}$, so ${\sf cl}(\min{\downarrow}y_\infty)=\{x_\infty\}$, and hence $y_\infty\notin{\sf cl}(\min{\downarrow}y_\infty)$. Consequently, \cref{I-subfit 6} is not satisfied, so $\I A$ is not subfit.
\end{proof}

Finally, we consider subfitness for the canonical completion.

\begin{proposition} \label{prop: canonical subfit iff boolean}
    For $A\in{\DLat}$, $A\!^\sigma$ is subfit iff $A$ is Boolean.
\end{proposition}

\begin{proof}
    First suppose $A$ is Boolean. Then the order on the Priestley space $X$ of $A$ is trivial, so ${\sf Up}(X)=\mathcal P(X)$. By \cref{lem: char of M and Dinfty-0}, $A\!^\sigma$ is Boolean, hence subfit. 
   Conversely, suppose $A\!^\sigma$ is subfit, so ${\sf Up}(X)$ is subfit. Thus, the order on $X$ is trivial by \cite[Cor.~4.9]{erne_complete_2007}. Consequently, $A$ is Boolean (see, e.g., \cite[p.~119]{gratzer_lattice_2011}).  
    \end{proof}

\begin{remark}
    We recall that a space $X$ is {\em Alexandroff} if an arbitrary intersection of open sets is open, in which case the opens of $X$ are exactly the upsets of $X$ in the specialization preorder (defined by $x\le y$ iff $x\in{\sf cl}(y)$). An Alexandroff space is $T_0$ iff the specialization preorder is a partial order. Using this language, \cite[Cor.~4.9]{erne_complete_2007} states that the frame of opens of an Alexandroff $T_0$-space is subfit iff $X$ is discrete. This result generalizes to arbitrary Alexandroff spaces as follows: 
    
    Let $X$ be an Alexandroff space and $X_0$ its $T_0$-reflection (that is, $X_0$ is the quotient of $X$ by the equivalence relation $\sim$ defined by $x\sim y$ iff ${\sf cl}(x)={\sf cl}(y)$). Since the frames $\O(X)$ and $\O(X_0)$ of opens of $X$ and $X_0$ are isomorphic, we have that $\O(X)$ is subfit iff $\O(X_0)$ is subfit. Moreover, the specialization order on $X_0$ is discrete iff the specialization preorder on $X$ is an equivalence relation. Thus, \cite[Cor.~4.9]{erne_complete_2007} yields: $\O(X)$ is subfit iff the specialization preorder of $X$ is an equivalence relation.
\end{remark}
As a consequence of \cref{prop: canonical subfit iff boolean}, we obtain that subfitness is not preserved but is reflected by canonical completions:

\begin{corollary}
    \label{prop: subfitness A vs can}
Let $A\in\DLat$.
\begin{enumerate}[label=\upshape(\arabic*), ref = \theproposition(\arabic*)]
 \item $A\!^\sigma$ subfit implies $A$ subfit. \label[proposition]{prop: subfitness A vs can 2}
    \item $A$ subfit does not imply $A\!^\sigma$ subfit. \label[proposition]{prop: subfitness A vs can 1}
   
\end{enumerate}
\end{corollary}

\begin{proof} 
(1) If $A\!^\sigma$ is subfit, then $A$ is Boolean by \cref{prop: canonical subfit iff boolean}, and hence $A$ is subfit.

(2) Take any subfit non-Boolean $A$ (for example, the lattice consisting of $\varnothing$  and the cofinite subsets of $\N$) and apply \cref{prop: canonical subfit iff boolean}.
\end{proof}

The situation summarizes as follows:

\begin{summary} \label{sum: subfit}
    Let $A\in\DLat$.
    \begin{enumerate}[label=\upshape(\arabic*), ref = \thesummary(\arabic*)]
    \item $\BL A$ is subfit iff $\pH A$ is subfit iff $\min pX$ is dense in $pX$.
    \item $\M A$ is subfit iff $A$ is subfit iff $\min X$ is dense in $X$.
   
    \item $\I A$ is subfit iff $\min X$ is dense in the Skula topology of ${\sf OpUp}(X)$.
     \item $A\!^\sigma$ is subfit iff $A$ is Boolean iff $\min X =X$. \label[summary]{sum: subfit 4}    
        \end{enumerate}
\end{summary}
This motivates the following definition.
\begin{definition}\label{defn: BL and I subfit}
    We call $A\in \DLat$ \emph{$\BL$-subfit} if $\BL A$ is subfit, and \emph{$\I$-subfit} if $\I A$ is subfit.
\end{definition}
\begin{remark} Obviously the terms $\pH$-subfit, $\M$-subfit, and $(\cdot)^\sigma$-subfit are redundant;
 $\BL$-subfitness is a proper weakening of subfitness, while $\I$-subfitness is a proper
  strengthening thereof.
Subfitness is equivalent to every principal ideal being an intersection of maximal ideals. $\I$-subfitness strengthens this to all ideals, which is not quite equivalent to $A$ being Boolean (the same example as in \cref{prop: subfitness A vs can 1}
is $\I$-subfit --- as every ideal is principal --- but not Boolean). For $A$ to be Boolean,   $A\!^\sigma$ must be subfit.
\end{remark}

\section{Regularity} \label{sec: regularity}

In this section we characterize when the four completions are regular. For distributive lattices, regularity is a natural strengthening of subfitness. As we saw above, subfitness of the canonical completion is equivalent  to the lattice being Boolean. As we will see below, regularity yields the same result for the ideal completion. A more subtle result occurs for the Bruns-Lakser completion. The situation for the Dedekind-MacNeille completion is further complicated by the fact that the latter need not be distributive. Regular lattices behave very differently in the non-distributive setting; in particular, regularity no longer implies subfitness. This will be discussed  in more detail in a forthcoming paper. Here we will restrict our attention to Dedekind-MacNeille completions that are distributive.

The next definition is well known for frames (see, e.g., \cite[Sec.~V.5]{picado_frames_2012}) and generalizes directly to bounded lattices.
\begin{definition} \label{def: regular DLat}
Let $A$ be a bounded lattice. 
\begin{enumerate}
\item For $a,b \in A$, $a$ is {\em rather below} ({\em well inside}) $b$, written $a\rb b$, provided there is $c\in A$ such that $a\wedge c=0$ and $b\vee c=1$. 
  \item   $A$ is 
   {\em regular} if $a\not\le b$ implies the existence of $c\in A$ such that $c \rb a$ but $c \not\le b$. 
   \end{enumerate}
\end{definition}

\begin{samepage}
\begin{remark}\label{rem: pseudo and reg}\hfill
\begin{enumerate}  [ref = \theremark(\arabic*)]
\item If $a$ has the pseudocomplement $a^*$, then $a\rb b$ iff $a^* \vee b=1$. \label[remark]{rem: pseudo and reg 1} 
 \item $A$ is regular iff every element is the join of elements rather below it. This is the usual definition of regularity in frames, but the above definition is more appealing for arbitrary bounded lattices. Observe that while arbitrary joins may not exist in $A\in\DLat$, if $A$ is regular then the join of $\{ b\in A : b\rb a\}$ is $a$, so this join does exist for every $a\in A$. \label[remark]{rem: pseudo and reg 2}
    \item Every Boolean lattice is regular and in \cite[p.~360]{ball_lindelof_2017} it is shown that every regular distributive lattice is subfit. In a forthcoming paper we will show that this is no longer true for non-distributive lattices. 
\end{enumerate}     
\end{remark}
\end{samepage}

Our first goal is to characterize when distributive lattices are regular using Priestley duality. We do so by generalizing some existing results for frames, proved in \cite[Lem.~3.3,~3.6] {bezhanishvili_spectra_2016}. Since pseudocomplements may not exist in an arbitrary distributive lattice, the proofs below
are somewhat more involved.
We start by characterizing the rather below relation. 
\begin{lemma} \label{lem: rather below} 
For $A\in\DLat$ with Priestley dual $X$,
   \[ a \rb b\text{ iff }{\downarrow}\s(a)\subseteq\s(b).\]
\end{lemma}

\begin{proof}
    First suppose $a \rb b$. Then there is $c\in A$ such that $a\wedge c=0$ and $b\vee c=1$. Therefore, $\s(a)\cap\s(c)=\varnothing$ and $\s(b)\cup\s(c)=X$. Since $\s(c)$ is an upset, ${\downarrow}\s(a)\cap\s(c)=\varnothing$. Thus, ${\downarrow}\s(a) \subseteq X\setminus\s(c) \subseteq \s(b)$.

    Conversely, suppose ${\downarrow}\s(a)\subseteq\s(b)$. Since ${\downarrow}\s(a)$ is a closed downset, it is the intersection of clopen downsets containing it (see \cref{fact: Priestley}). But then by compactness, there is a clopen downset $D$ such that ${\downarrow}\s(a)\subseteq D\subseteq\s(b)$. Let $c\in A$ be such that $\s(c)=X\setminus D$. Then ${\downarrow}\s(a)\cap\s(c)=\varnothing$ and $X\setminus\s(b)\subseteq\s(c)$. Therefore, $\s(a)\cap\s(c)=\varnothing$ and $\s(b)\cup\s(c)=X$. Thus, $a\wedge c=0$ and $b\vee c=1$, and hence $a \rb b$.
\end{proof}

Following \cite[Def.~3.4]{bezhanishvili_spectra_2016}, for each $V\in {\sf L}(X)$, define the {\em regular part} of $V$ to be  
\[
R(V) = \bigcup \{ U \in {\sf L}(X) : {\downarrow}U\subseteq V \}.
\]
For simplicity, we abbreviate the {\em regular part} of $\s(a)$ by $R_a$.

\begin{proposition} \label{lem: regular} 
For $A\in\DLat$ with Priestley dual $X$, the following are equivalent:
\begin{enumerate}[label=\upshape(\arabic*), ref = \theproposition(\arabic*)]
    \item $A$ is regular.
    \item $R_a$ is dense in $\s(a)$ for each $a\in A$.
\end{enumerate}
  
\end{proposition} 

\begin{proof}
     (1)$\Rightarrow$(2): Let $a\in A$. Take $x\in\s(a)$ and let $U$ be a clopen neighborhood of $x$ in $\s(a)$. We must show that $U \cap R_a \ne \varnothing$. Since $U$ is an open neighborhood of $x$, there are $c,d\in A$ such that $x\in\s(c)\setminus\s(d)\subseteq U$ (see \cref{fact: Priestley 1}). Without loss of generality we may assume that $\s(c)\subseteq\s(a)$ (by intersecting $\s(c)$ with $\s(a)$). Because $s(c)\setminus\s(d)\ne\varnothing$, $\s(c)\not\subseteq\s(d)$, so $c\not\le d$. Since $A$ is regular, there is $e\in A$ such that $e \rb c$ and $e\not\le d$. By \cref{lem: rather below}, ${\downarrow}\s(e)\subseteq\s(c)$, so $\s(e)\subseteq R_c\subseteq R_a$ and $\s(e)\not\subseteq\s(d)$. Thus, $R_a\cap(\s(c)\setminus\s(d))\ne\varnothing$, so $R_a\cap U\ne\varnothing$, and hence $R_a$ is dense in $\s(a)$. 

    (2)$\Rightarrow$(1): Let $a\not\le b$. Then $\s(a)\not\subseteq\s(b)$, so $R_a\not\subseteq\s(b)$ since $R_a$ is dense in $\s(a)$. Therefore, there is $c\in A$ such that ${\downarrow}\s(c)\subseteq\s(a)$ and $\s(c)\not\subseteq\s(b)$. Thus, $c \rb a$ (by \cref{lem: rather below}) and $c\not\le b$, yielding that $A$ is regular. 
\end{proof}

 We next give a general characterization, akin to \cref{thm: Janowitz}, of when regularity is preserved and/or reflected when moving between $A$ and $B$, where $A$ is a bounded sublattice of $B$. We write $\rb_A$ or $\rb_B$ to emphasize which rather below relation is being used. Clearly $a \rb_A b$ implies $a \rb_B b$ for all $a,b\in A$, however the converse is not true in general. Because of this, the proposition below requires an extra assumption.  As in  \cref{thm: Janowitz}, distributivity plays no r\^ole.

\begin{proposition} \label{prop: when reg is preserved}
    Let $A$ be  a bounded sublattice of $B$. Suppose that $a \rb_A b$ iff $a \rb_B b$ for all $a,b\in A$.
    \begin{enumerate}[label=\normalfont(\arabic*), ref = \thelemma(\arabic*)]
        \item If $A$ is join-dense in $B$  and $B$ is regular then $A$ is regular.
        \item If $A$ is both join- and meet-dense in $B$ and $A$ is regular then $B$ is regular.
    \end{enumerate}
\end{proposition}

\begin{proof}
    (1) Let $a,b\in A$ with $a\not\le b$. Since $B$ is regular, there is $x\in B$ with $x \rb_B a$ and $x\not\le b$. Because $A$ is join-dense in $B$, there is $c\in A$ with $c\le x$ and $c\not\le b$. Therefore, $c\le x \rb_B a$, so $c \rb_B a$, and hence $c \rb_A a$ (by assumption) and $c\not\le b$. Thus, $A$ is regular.

    (2) Let $x,y\in B$ with $x\not\le y$. Since $A$ is both join- and meet-dense in $B$, there are $a,b\in A$ such that $a\le x$, $y\le b$, and $a\not\le b$. Because $A$ is regular, there is $c\in A$ with $c \rb_A a$ and $c\not\le b$. Therefore, 
   $c \rb_B a$ (by assumption) and $c\not\le b$, so $c \rb_B x$  and $c\not\le y$. Thus, $B$ is regular.
\end{proof}

Since $A$ is both join- and meet-dense in $\M A$, one could expect that regularity is both preserved and reflected by the Dedekind-MacNeille completion, but 
this is not the case
since the additional assumption in \cref{prop: when reg is preserved} may not be satisfied. We again emphasize that $\M A$ is not always distributive and hence regularity  of $\M A$ need not imply subfitness. Thus, we make the blanket assumption that $\M A$ is distributive. While this does not necessarily imply that $\M A$ is a frame, it is indeed the case provided $\M A$ is regular.

\begin{proposition} \label{prop: reg MA is reg BLA}
Let $A\in\DLat$ and $\M A $ be  distributive. Then regularity of $\M A$ implies that $\M A$ is a frame, and hence $\M A \cong \BL A$.
\end{proposition}

\begin{proof}
If $\M A$ is distributive, by \cite[p.~360]{ball_lindelof_2017} regularity of $\M A$ implies subfitness of $\M A$, and consequently subfitness of $A$ by \cref{thm: MA subfit implies A subfit}. But then $A$ is proHeyting by \cref{prop: subfit implies proH}. Hence, $\M A$ is a frame and $\M A \cong \BL A$ by \cref{thm: MA frame}.
\end{proof}
The above proposition yields the following characterization of when a distributive Dede\-kind-MacNeille completion is regular.

\begin{proposition} \label{prop: reg MA }
Let $A\in\DLat$ and $\M A $ be distributive. The following are equivalent:
\begin{enumerate}[label=\upshape(\arabic*), ref = \theproposition(\arabic*)]
    \item $\M A$ is regular.
    \item $A$ is subfit and $\BL A$ is regular.
    \item $A$ is proHeyting and $\BL A$ is regular.
\end{enumerate}
\end{proposition}

\begin{proof}
(1)$\Rightarrow$(2): That $\BL A$ is regular follows from \cref{prop: reg MA is reg BLA}. Regular implies subfit for distributive lattices, so $\M A$ is subfit, but then so is $A$ by \cref{thm: MA subfit implies A subfit}.

(2)$\Rightarrow$(3): This follows from \cref{prop: subfit implies proH}. 

(3)$\Rightarrow$(1): This follows from \cref{thm: MA frame}.
\end{proof}

We next concentrate on characterizing when $\BL A$ is regular.
Since $\BL A$
is a frame, the rather below relation is conveniently expressible in terms of the pseudocomplement (see \cref{rem: pseudo and reg 1}). We use Priestley duality to describe this pseudocomplement. By \cref{lem: char of M and Dinfty-3}, $\BL A \cong \sf{BL} (X)$ and by \cref{fact: Priestley 5}, $\sf{BL} (X)$ is the sublocale of ${\sf OpUp}(X)$ given by the nucleus ${\sf int_1 cl}$.
Since the pseudocomplement in ${\sf OpUp}(X)$ is given by 
\[
U^* = X \setminus{\downarrow}{\sf cl}(U),
\]
the pseudocomplement in ${\sf BL}(X)$ is given by
\[
\lnot U = {\sf int_1 cl}(U^*) = {\sf int_1 cl}(X \setminus{\downarrow}{\sf cl}(U)).
\]

\begin{lemma} \label{lem: rather below in BL}
    Let $U,V\in{\sf BL}(X)$. 
    \begin{enumerate}[label=\upshape(\arabic*), ref = \thelemma(\arabic*)]
       \item $ {\sf int} \, {\downarrow} \, {\sf cl \, int} \,{\downarrow} \, {\sf cl}(U)=
         {\sf int} \, {\downarrow} \, {\sf cl} (U)$.
        \item $U{\rb_{{\sf BL}}}V$ iff ${\sf int}{\downarrow}{\sf cl}(U)\subseteq{\sf cl}(V)$.
        \item If $U,V\in{\sf L}(X)$, then $U{\rb_{{\sf BL}}}V$ iff ${\sf int}{\downarrow}U\subseteq V$. \label[lemma]{lem: rather below in BL 2}
    \end{enumerate}  
\end{lemma}

\begin{proof} (1) The right-to-left inclusion is obvious. For the other inclusion, observe that 
\[
{\sf int} \, {\downarrow} \, {\sf cl \, int} \,{\downarrow} \, {\sf cl}(U) \subseteq {\sf int} \, {\downarrow} \, {\sf cl} \,{\downarrow} \, {\sf cl}(U) \subseteq {\sf int} \, {\downarrow} \, {\sf cl}(U)
\]
because ${\downarrow} \, {\sf cl}(U)$ is a closed downset (see \cref{fact: Priestley 4}). 

   (2) Recalling how joins are defined in ${\sf BL}(X)$ (see \cref{fact: Priestley 5}) and that ${\sf cl_1}={\downarrow} \, {\sf cl}$, by (1) we have:
    \begin{eqnarray*}
        U{\rb_{{\sf BL}}}V & \Leftrightarrow & \lnot U\vee V=X \\
        & \Leftrightarrow & {\sf cl}(\lnot U)\cup{\sf cl}(V)=X \\
        & \Leftrightarrow & {\sf cl\, int_1 cl}(X\setminus {\downarrow}{\sf cl} (U))\cup {\sf cl}(V)=X \\
        & \Leftrightarrow & (X \setminus {\sf int\, cl_1 \, int} \, {\downarrow}\, {\sf cl}(U)) \cup {\sf cl}(V)=X \\
        & \Leftrightarrow & {\sf int} \, {\downarrow} \, {\sf cl \, int} \,{\downarrow} \, {\sf cl}(U)\subseteq {\sf cl}(V) \\
        & \Leftrightarrow & {\sf int} \, {\downarrow} \, {\sf cl} (U)\subseteq {\sf cl}(V).
    \end{eqnarray*}  
    
      (3) This follows from (2) since $U,V\in{\sf L}(X)$ imply ${\sf cl}(U)=U$ and ${\sf cl}(V)=V$.
   \end{proof}

This allows us to see that when restricted to $A$, the rather below relations on $\BL A$ and $\pH A$ coincide.
\begin{proposition}\label{prop: rb in BL and pH}
For each $a,b \in A$, $a \rb_{\BL A}b$ iff $a \rb_{\pH A}b$.
\end{proposition}
\begin{proof}
First suppose that $a \rb_{\BL A}b$. Let $X$ be the Priestley space of $A$. Then ${\sf int}{\downarrow}\s(a)\subseteq\s(b)$ by \cref{lem: rather below in BL 2}. Let $U=X\setminus {\downarrow}\s(a)$. Then $U\in{\sf pH}(X)$ by \cref{lem: annih char}, and $\s(a)\cap U=\varnothing$. Moreover, since ${\sf int}{\downarrow}\s(a)\subseteq\s(b)$, we have 
    \[
    {\sf cl}(U)\cup\s(b)={\sf cl}(X\setminus{\downarrow}\s(a))\cup\s(b)=(X\setminus{\sf int}{\downarrow}\s(a))\cup\s(b)=X.
    \]
    Therefore, $U\vee_{{\sf pH}(X)}\s(b)=X$, so $\s(a)$ is rather below $\s(b)$ in ${\sf pH}(X)$, and hence $a\rb_{\pH A}b$ by \cref{prop: pH A = pH X}. 

    For the converse, suppose that $a\rb_{\pH A}b$. By \cref{prop: pH A = pH X}, there is $U\in{\sf pH}(X)$ such that $\s(a)\cap U=\varnothing$ and $U \vee_{{\sf pH}(X)}\s(b)=X$. The former gives ${\downarrow}\s(a) \cap U = \varnothing$, so $U \subseteq X\setminus{\downarrow}\s(a)$. The latter yields that ${\sf cl}(U)\cup\s(b)=X$. Therefore, ${\sf cl}(X\setminus{\downarrow}\s(a))\cup\s(b)=X$, so $(X\setminus{\sf int}{\downarrow}\s(a))\cup\s(b)=X$, and hence ${\sf int}{\downarrow}\s(a)\subseteq\s(b)$. Thus, $a\rb_{\BL A}b$ by \cref{lem: rather below in BL 2}.
    \end{proof}

In order to show that regularity of $\BL A$ is equivalent to that of $\pH A$, we require the following general fact.

\begin{proposition} \label{BL regular} 
Let $A$ be a bounded sublattice of $B$. If $A$ is join-dense in $B$, then the following are equivalent:
\begin{enumerate}[label=\upshape(\arabic*), ref = \thetheorem(\arabic*)]
    \item $B$ is regular.
    \item $b=\bigvee\{ c\in A : c \rb_{B}b \}$ for each $b\in B$.
    \item $a=\bigvee\{ c\in A : c \rb_{B}a\}$ for each $a\in A$.  \label[proposition]{BL regular 3}
  \end{enumerate}  
\end{proposition}

\begin{proof}
(1)$\Rightarrow$(2): Let $b\in B$. Since $B$ is regular, $b=\bigvee\{ d\in B : d \rb_{B}b \}$ (see \cref{rem: pseudo and reg 2}). Because $A$ is join-dense in $B$, $d = \bigvee \{ c \in A : c \le d \}$. But $c \le d \rb_B b$ implies $c \rb_B b$, so $b=\bigvee\{ c\in A : c \rb_{B}b \}$.

(2)$\Rightarrow$(3): This is obvious. 

(3)$\Rightarrow$(1) Let $b\in B$. Since $A$ is join-dense in $B$, $b = \bigvee \{ a \in A : a \le b \}$. By assumption, $a = \bigvee \{ c\in A : c \rb_{B}a\}$ for each $a \
\in A$. From $c \rb_B a \le b$ it follows that $c \rb_B b$. Therefore, $b=\bigvee\{c\in A : c \rb_B b\}$, and hence $B$ is regular.
\end{proof}

\begin{theorem} \label{thm: BL A reg = pH A reg}
    For $A\in\DLat$, $\BL A$ is regular iff $\pH A$ is regular.
\end{theorem}

\begin{proof}
Let $a\in A$. Since $\pH A$ is join-dense in $\BL A$, it follows from \cref{prop: rb in BL and pH} that $a=\bigvee_{\BL A}\{ c\in A : c \rb_{\BL A}a\}$ iff $a=\bigvee_{\pH A}\{ c\in A : c \rb_{\pH A}a\}$. But $A$ is join-dense in both $\pH A$ and $\BL A$, and thus, by \cref{BL regular}, $\BL A$ is regular iff $\pH A$ is regular.
\end{proof}

We next give a characterization of when $\BL A$ is regular akin to \cref{lem: regular}.

\begin{definition}
    Let $V\in{\sf BL}(X)$. We call 
\[
R_{\sf BL}(V) := \bigcup\{ U\in{\sf L}(X) : U {\rb_{{\sf BL}}} V \}
\]
the {\em regular part of $V$ in ${\sf BL}(X)$}.
\end{definition}

\begin{theorem} \label{BL regular1}
Let $A\in \DLat$. The following are equivalent:
\begin{enumerate}[label=\upshape(\arabic*), ref = \thetheorem(\arabic*)]
    \item $\BL A$ is regular.
     \item $R_{\sf BL}(V)$ is dense in $V$ for each $V\in{\sf BL}(X)$.
    \item $R_{\sf BL}(V)$ is dense in $V$ for each $V\in{\sf L}(X)$. \label[theorem]{BL regular 5}   
    \end{enumerate}
\end{theorem}

\begin{proof} 
(1)$\Leftrightarrow$(2): Since $\BL A \cong {\sf BL}(X)$, it is enough to show that ${\sf BL}(X)$ is regular iff $R_{\sf BL}(V)$ is dense in $V$ for each $V\in{\sf BL}(X)$. 
    Let $V\in{\sf BL}(X)$. We have 
    \[
    \bigvee\{ U\in{\sf L}(X) : U {\rb_{{\sf BL}}} V \} = {\sf int_1 \, cl}\bigcup\{ U\in{\sf L}(X) : U {\rb_{{\sf BL}}} V \} = {\sf int_1 \, cl}(R_{\sf BL}(V)).
    \]
    Therefore, by \cref{BL regular}, ${\sf BL}(X)$ is regular iff $V=
   {\sf int_1 \, cl}(R_{\sf BL}(V))$ iff $R_{\sf BL}(V)$ is dense in $V$ for each $V\in{\sf BL}(X)$. 

(1)$\Leftrightarrow$(3): By \cref{BL regular}, ${\sf BL}(X)$ is regular iff $V = \bigvee \{ U \in{\sf L}(X) : U {\rb_{{\sf BL}}} V \}$ for each $V \in {\sf L}(X)$. The same proof as in (1)$\Leftrightarrow$(2) may be used to show that this is equivalent to $R_{\sf BL}(V)$ being dense in $V$ for each  $V \in {\sf L}(X)$. 
\end{proof}

We now compare regularity of $A$ to that of $\BL A$.

\begin{theorem}\label{thm: BL preserves regular}
Let $A\in{\DLat}$. 
\begin{enumerate}[label=\upshape(\arabic*), ref = \thetheorem(\arabic*)]
    \item $A$ regular implies that $\BL A$ is regular. \label[theorem]{thm: BL preserves regular a}
    \item  $\BL A$ regular does not imply that $A$ is regular. \label[theorem]{thm: BL preserves regular b}
\end{enumerate}  
\end{theorem} 

\begin{proof} 

(1) This follows from \cref{BL regular} since $a \rb_A b$ implies $a \rb_{\BL A} b$ for each $a,b \in A$.\footnote{It can also be derived from \cref{lem: regular,BL regular 5} since $R(U)\subseteq R_{\sf BL}(U)$ for each clopen upset $U$ of the Priestley space of $A$.}

(2) Let $A$ be the finite subsets of $\N$ together with $\N$. Then $A \in \DLat$, but it is not regular since it is not subfit. 
 On the other hand, we show that $\BL A$ is isomorphic to $\mathcal P(\N)$, so $\BL A$ is Boolean, hence regular. Let $X$ be the Priestley space of $A$, where $x_n$ is the prime filter of $A$ generated by $\{n\}$ and $x_\infty$ is $\{\N\}$ (see \cref{Fig 3}).  

  \begin{center} 
    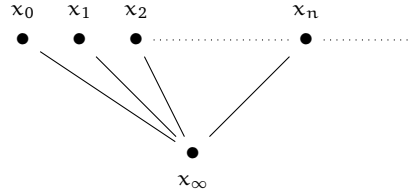
\begin{figure}[hbt!]
\begin{tikzpicture}[scale=.75]
 \node  (x0) at (0,1) {$\bullet$};
         \node  (x0_) at (0,1.5) {$\scriptstyle{x_0}$}; 
\node  (x1) at (1,1) {$\bullet$};\node  (x1_) at (1,1.5) {$\scriptstyle{x_1}$}; 
\node  (x2) at (2,1) {$\bullet$};\node  (x2_) at (2,1.5) {$\scriptstyle{x_2}$}; 
\node  (xn) at (5,1) {$\bullet$};\node  (xn_) at (5,1.5) {$\scriptstyle{x_n}$}; 
\node  (xinf) at (3,-1) {$\bullet$};\node  (xinf_) at (3,-1.5) {$\scriptstyle{x_\infty}$};

\node  (y1) at (7,1){};
\draw (x0) --  (xinf);
\draw (x1) --  (xinf);
\draw (x2) --  (xinf);
\draw (xn) --  (xinf);
\draw[dotted] (x2) --  (xn);
\draw[dotted] (xn) --  (y1);
\end{tikzpicture}
\caption{The Priestley space $X$ of $A$.}
        \label{Fig 3}
    \end{figure}
\end{center}
 
 By \cref{lem: char of M and Dinfty-3}, it is enough to show that ${\sf BL}(X)$ is isomorphic to $\mathcal P(\N)$. Since $x_\infty$ is the only limit point of $X$, open upsets of $X$ are $X$ and the subsets of $\{ x_n : n\in\N \}$. Of these, only $\{ x_n : n\in\N \}$ is not a BL-upset, yielding that ${\sf BL}(X)$ is isomorphic to $\mathcal P(\N)$.
 \end{proof}

We next strengthen \cref{thm: BL preserves regular b} and show that there are even subfit $A$ such that $\BL A$ is regular but $A$ is not. 

\begin{example} \label{ex: BL A reg A not reg}
Consider the space $X$ depicted in \cref{Figure 4}, where 
\begin{itemize}
    \item $\{x_n : n\in\N \}\cup\{ y_n : n\in\N \}\cup\{z\}$ is the set of isolated points; 
    \item $x_\infty$ is the limit of $\{x_n : n\in\N \}$; 
    \item $y_\infty$ is the limit of $\{y_n : n\in\N \}$. 
\end{itemize}
Thus, $X$ is the two-point compactification of the discrete space 
\[
\{x_n : n\in\N \}\cup\{ y_n : n\in\N \}\cup\{z\},
\]
and hence is a Stone space. 
It is straightforward to check that with the order as depicted below, $X$ is a Priestley space. 

 \begin{center}
    \begin{figure}[hbt!]
\begin{tikzpicture}[scale=.75]
 \node  (x0) at (0,1) {$\bullet$};
         \node  (x0_) at (0,0.5) {$\scriptstyle{x_0}$}; 
\node  (x1) at (1,1) {$\bullet$};\node  at (1,0.5) {$\scriptstyle{x_1}$}; 
\node  (x2) at (2,1) {$\bullet$};\node   at (2,0.5) {$\scriptstyle{x_2}$}; 
\node  (z) at (6,1) {$\bullet$};\node   at (6,0.5) {$\scriptstyle{z}$}; 
\node  (xinf) at (5,1) {$\bullet$};\node at (5,0.5) {$\scriptstyle{x_\infty}$};
\node  (yinf) at (5.5,3) {$\bullet$};\node  at (5.5,3.5) {$\scriptstyle{y_\infty}$};
\node  (y2) at (8,3) {$\bullet$};\node   at (8,3.5) {$\scriptstyle{y_2}$};
\node  (y1) at (9,3) {$\bullet$};\node   at (9,3.5) {$\scriptstyle{y_1}$};
\node  (y0) at (10,3) {$\bullet$};\node   at (10,3.5) {$\scriptstyle{y_0}$};
\draw[dotted] (x2) --  (xinf);
\draw[dotted] (yinf) --  (y2);
\draw  (xinf)-- (yinf);
\draw  (z)-- (yinf);
\end{tikzpicture}
\caption{A Priestley space $X$ with ${\sf L}(X)$ subfit, ${\sf BL}(X)$ regular, but ${\sf L}(X)$ not regular.}
        \label{Figure 4}
    \end{figure}
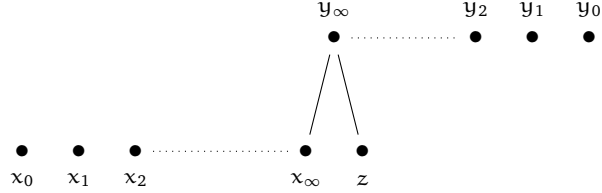
\end{center}

Let $A = {\sf L}(X)$ and, as before, identify $X$ with the Priestley space of $A$. Since $\min X$ is dense in $X$ (the only non-minimal point is $y_\infty$, which is a limit point), we see that $A$ is subfit (by \cref{prop: char of subfit}).
But $A$ is not regular. To see this, consider the clopen upset $U = \{ y_1, y_2, \dots \} \cup \{ y_\infty,z\}$. We show that $R(U)=\{ y_1, y_2, \dots \}$. Let $V$ be a clopen upset contained in $U$. If $y_\infty \in V$, then ${\downarrow}V \not\subseteq U$ since $x_\infty\in{\downarrow}V\setminus U$. Otherwise 
$V\subseteq\{ y_1, y_2, \dots \}$, and hence ${\downarrow}V=V \subseteq U$. Thus, $V \prec U$ iff $V \subseteq \{ y_1, y_2, \dots \}$, yielding that $R(U)=\{ y_1, y_2, \dots \}$. Since $z \notin {\sf cl}\, R(U)$, we see that $R(U)$ is not dense in $U$. Therefore, $A$ is not regular by \cref{lem: regular}. 

On the other hand, we show that $\BL A$ is regular. For this, by \cref{BL regular1}, it is sufficient to show that $R_{\sf BL}(V)$ is dense in $V$ for each $V\in{\sf L}(X)$. If $y_\infty\notin V$, then $V$ is a finite subset of $\{x_n : n\in\N \}\cup\{ y_n : n\in\N \}$, so $R(V)=V$, and hence $R_{\sf BL}(V)=V$. 

Suppose $y_\infty \in V$. Then cofinitely many $y_n$ are in $V$. We have that $z\in V$ or $z\notin V$. First let $z\in V$.
If $x_\infty \notin V$, then $V$ only contains finitely many $x_n$, so ${\sf int}\,{\downarrow}\,V = V$, and hence $R_{\sf BL}(V)=V$.
If $x_\infty\in V$, then $V$ is a downset, so $R(V)=V$, and thus $R_{\sf BL}(V)=V$.

Next let $z\notin V$. If $x_\infty \notin V$, then $V$ only contains finitely many $x_n$, so $R_{\sf BL}(V)$ consists of these finitely many $x_n$ and cofinitely many $y_n$, and hence $R_{\sf BL}(V)$ is dense in $V$. 
If $x_\infty\in V$, then cofinitely many $x_n$ are in $V$, so $R_{\sf BL}(V)$ consists of cofinitely many $x_n$ and cofinitely many $y_n$, and again $R_{\sf BL}(V)$ is dense in $V$. 
Thus, in all cases we have that $R_{\sf BL}(V)$ is dense in $V$.
\end{example}

We now introduce a stronger notion of regularity for $\BL A$
which is equivalent to $A$ being regular. This notion makes sense for any lattice $B$ that has $A$ as a bounded sublattice.

\begin{definition} 
Let $A$ be a bounded sublattice of $B$.
    \begin{enumerate}
        \item For $a\in A$ and $b\in  B$, $a$ is \emph{$A$-below} $b$, written $a\lhd b$, provided there is $c\in A$ such that $a\prec_A c \le b$. 
        \item We say that $B$ is {\em $A$-regular} if for each $b,d\in B$, from $b\not\le d$ it follows that there is $a \in  A$ with $a \lhd b$ and $a\not\le d$.
    \end{enumerate}
\end{definition}

\begin{remark}\ \label{rem: lhd vs rb}
\begin{enumerate}[label=\upshape(\arabic*), ref = \theremark(\arabic*)]
    \item Clearly, if $a,b \in A$ then $a \lhd b$ iff $a \rb_A b$, and $B$ is $A$-regular iff $b =\bigvee \{ a \in A : a \lhd b \}$ for each $b\in B$ (cf.~\cref{BL regular}). \label[remark]{rem: lhd vs rb 1}
    \item If $A$ is join-dense in $B$ then $B$ is $A$-regular iff the following condition holds: 
    \[
    \mbox{If $a\in A$, $b\in B$, and $a\not\le b$ then there is $c \in  A$ with $c \rb_A a$ and $c\not\le b$.}
    \]
\end{enumerate}
\end{remark}

\begin{proposition} \label{prop: strongly regular}
Let $A$ be a join-dense bounded sublattice of $B$. Then
    $A$ is regular iff $B$ is $A$-regular.
\end{proposition}

\begin{proof}
    Suppose $A$ is regular and $b\in B$. Since $A$ is join-dense in $B$, $b=\bigvee\{a \in A : a \le b\}$.  But $A$ is regular so $a=\bigvee\{ c\in A : c \prec_A a \}$ for each $a\in A$. From $c\rb_A a\le b$ it follows that $c\lhd b$. Therefore, $b=\bigvee\{c\in A : c\lhd b\}$, and thus $B$ is $A$-regular. 
    
    Conversely, suppose $B$ is $A$-regular and $a,b\in A$ with $a\not\le b$. Since $a=\bigvee\{ c\in A : c\lhd a \}$, there is $c\in A$ such that $c\lhd a$ and $c\not\le b$. Therefore, $A$ is regular by \cref{rem: lhd vs rb 1}. 
\end{proof}

Since $A\in\DLat$ is join-dense in both $\pH A$ and $\BL A$, we conclude:

\begin{corollary}
    For $A\in\DLat$, the following are equivalent:
    \begin{enumerate}[label=\upshape(\arabic*), ref = \thecorollary(\arabic*)]
        \item $A$ is regular.
        \item $\pH A$ is $A$-regular.
        \item $\BL A$ is $A$-regular.
    \end{enumerate}
\end{corollary}

As another consequence, we obtain:

\begin{corollary}
    If $\BL A$ (resp.~$\pH A$) is $A$-regular, then $\BL A$ (resp.~$\pH A$) is regular, but not conversely.
\end{corollary}

\begin{proof}
    If $\BL A$ is $A$-regular, then $A$ is regular by \cref{prop: strongly regular}, and hence $\BL A$ is regular by \cref{thm: BL preserves regular a}. On the other hand, let $A$ be as in \cref{ex: BL A reg A not reg}. Then $\BL A$ is regular, but $A$ is not regular. Therefore, $\BL A$ is not $A$-regular by \cref{prop: strongly regular}.
    \end{proof}

The remainder of this section focuses on characterizing when the ideal and canonical completions are regular. For this, we  recall some definitions and results from frame theory. For any frame $L$:
\begin{itemize}
    \item an element $a \in L$ is \emph{compact} if $a\le\bigvee S$ implies $a\le\bigvee T$ for some finite $T\subseteq S$;
    \item \cite[p.~63]{johnstone_stone_1982} $L$ is \emph{coherent} if the set $K(L)$ of compact elements is join-dense and a bounded sublattice of $L$;
    \item \cite{banaschewski_universal_1989} $L$ is \emph{Stone} if $L$ is coherent and $K(L)$ is a Boolean algebra;
    \item A coherent frame is Stone iff it is regular.\footnote{This was proved by Banaschewski in a seminar series at the University of Cape Town in 1988. For a proof, see \cite[Prop.~1.22]{walters_uniform_1989}, where the result is shown for $\sigma$-frames (an analogous argument works for frames).}
\end{itemize}

Recall that, up to isomorphism, coherent frames are exactly the frames of ideals of bounded distributive lattices, which in turn are realized as the compact elements of coherent frames (\cite[p.~64]{johnstone_stone_1982}). Thus, we obtain: 

\begin{proposition}\label{I-reg is Boolean}
    For $A\in{\DLat}$, $\I A$ is regular iff $A$ is Boolean.
\end{proposition}

\begin{proof}
    First suppose that $\I A$ is regular. Since $\I A$ is coherent, it follows that $\I A$ is Stone, hence the compact elements of $\I A$ form a Boolean algebra. But $K(\I A)$ is isomorphic to $A$, so $A$ is Boolean. Conversely, if $A$ is Boolean then it is well known that $\I A$ is Stone \cite{banaschewski_universal_1989}, hence regular.
\end{proof}
Transitioning to the canonical completion, since regularity implies subfitness for distributive lattices, one implication of the next proposition is a consequence of \cref{prop: canonical subfit iff boolean}, while the other follows from the fact that $A$ is Boolean iff  $A\!^\sigma$ is Boolean (see \cref{prop: Asigma boolean}):
\begin{proposition} \label{can-reg is Boolean}
    For $A\in{\DLat}$, $A\!^\sigma$ is regular iff $A$ is Boolean.
\end{proposition}

The situation summarizes as follows: 

\begin{summary}
    Let $A\in\DLat$.
    \begin{enumerate}
    \item If $\M A$ is distributive, then $\M A$ is regular iff $\BL A$ is regular and $A$ is subfit.
     \item $\BL A$ (resp.~$\pH A$) is regular iff $R_{\sf BL}(U)$ is dense in $U$ for each $U \in {\sf L}(X)$.  
      \item $\BL A$ (resp.~$\pH A$) is $A$-regular iff $A$ is regular iff $R(U)$ is dense in $U$ for each $U \in {\sf L}(X)$. 
     \item $\I A$ is regular iff $A\!^\sigma$ is regular iff $A$ is Boolean iff $R(U)=U$ for each $U \in {\sf L}(X)$.\footnote{For the last equivalence it is sufficient to observe that $R(U)=U$ for each $U \in {\sf L}(X)$ iff each $U$ is a complemented element of ${\sf L}(X)$. }
        \end{enumerate}
\end{summary}

\section{Booleanness}
We finally concentrate on when the four completions are Boolean. Janowitz \cite[Thm.~3.11]{janowitz_section_1968} gives a characterization for the Dedekind-MacNeille completion to be Boolean. We obtain his result as a corollary to our characterization of when the Bruns-Lakser completion is Boolean. This requires using the notion of $\wedge$-subfitness (see the introduction). 
Since $A$ is $\wedge$-subfit iff its order-dual is $\join$-subfit, \cref{prop: char of subfit} immediately yields: 

\begin{samepage}
\begin{proposition} \label{prop: char of cosubfit}
For $A\in{\DLat}$ and $X$ its Priestley dual, the following are equivalent:
    \begin{enumerate}[label=\upshape(\arabic*), ref = \theproposition(\arabic*)]
        \item $A$ is $\wedge$-subfit.
        \item $a\not\le b$ implies there is a maximal filter $F$ with $a \in F$ and $b \notin F$. \label[proposition]{prop: char of cosubfit 2}
        \item Every principal filter is an intersection of maximal filters. \label[proposition]{prop: char of cosubfit 2a}
       \item $\max X$ is dense in $X$. \label[proposition]{prop: char of cosubfit 3}
    \end{enumerate}
\end{proposition}
\end{samepage}

\begin{corollary}\label{cosubfit is Boolean}
Let $A\in{\DLat}$ with Priestley space $X$.
\begin{enumerate}[label=\upshape(\arabic*), ref = \thecorollary(\arabic*)]
    \item If  $\max X$ is closed, then $A$ is $\wedge$-subfit iff  $A$ is Boolean.
\item If $A$ is a frame, then $A$ is $\wedge$-subfit iff  $A$ is Boolean.\label[corollary]{cosubfit is Boolean 2}
    \end{enumerate}
\end{corollary}
\begin{proof}
    (1) Suppose $A$ is $\wedge$-subfit. By \cref{prop: char of cosubfit}, $\max X$ is dense in $X$. But $\max X$ is closed, so $\max X =X$, and hence 
$A$ is Boolean (see, e.g., \cite[p.~119]{gratzer_lattice_2011}). The other implication is always true.

(2) If $A$ is a frame then $\max X$ is closed (see, e.g., \cite[p.~47]{esakia_heyting_2019}) and hence (1) applies.
\end{proof}
We utilize the above two results to characterize when the Bruns-Lakser completion of a distributive lattice is Boolean.

\begin{theorem} \label{thm: BL boolean}
    For $A\in{\DLat}$, the following are equivalent.
    \begin{enumerate}[label=\upshape(\arabic*), ref = \thetheorem(\arabic*)]
        \item $\BL A$ is Boolean.
        \item $A$ is $\wedge$-subfit.
        \item $\pH A$ is Boolean.
        \end{enumerate} 
\end{theorem}

\begin{proof}
    (1)$\Rightarrow$(2): Let $a,b\in A$ with $a\not\le b$. Since $\BL A$ is Boolean, it is $\wedge$-subfit, so there is $u\in \BL A$ such that $b\wedge u=0$ and $a\wedge u\ne 0$. Because $A$ is join-dense in $\BL A$, $u=\bigvee S$ for some $S\subseteq A$. Therefore, since $\BL A$ is a frame, $b\wedge s=0$ for each $s\in S$ and there is $t\in S$ with $a\wedge t\ne 0$. Thus, $b\wedge t=0$ but $a\wedge t\ne 0$, and hence $A$ is $\wedge$-subfit.

    (2)$\Rightarrow$(3): Since $\pH A$ is generated as a distributive lattice by relative annihilators, it is enough to show that each relative annihilator has a complement in $\pH A$. By \cref{lem: annih char 1,prop: pH A = pH X}, it is sufficient to show that $X\setminus{\downarrow}(U\setminus V)$ has a complement in ${\sf pH}(X)$ for any clopen upsets $U,V$ of $X$. Specifically, we show that $X\setminus{\downarrow}(U^c\cup V)$ is the complement of $X\setminus{\downarrow}(U\setminus V)$ in ${\sf pH}(X)$. Since $U^c\cup V$ is clopen, $X\setminus{\downarrow}(U^c\cup V) \in {\sf pH}(X)$ by \cref{lem: annih char 2}. Moreover, it is straightforward to see that  
    \[
    [X\setminus{\downarrow}(U\setminus V)] \cap [X\setminus{\downarrow}(U^c\cup V)] = \varnothing.
    \]
    Furthermore, since $A$ is $\wedge$-subfit, $\max X$ is dense in $X$ by \cref{prop: char of cosubfit}. But 
    \[
    \max X \subseteq [X\setminus{\downarrow}(U\setminus V)] \cup [X\setminus{\downarrow}(U^c\cup V)],
    \]
    so $[X\setminus{\downarrow}(U\setminus V)] \cup [X\setminus{\downarrow}(U^c\cup V)]$ is dense in $X$, and hence 
    \[
    [X\setminus{\downarrow}(U\setminus V)] \vee_{{\sf pH}(X)} [X\setminus{\downarrow}(U^c\cup V)] = X.
    \]
    Thus, $X\setminus{\downarrow}(U\setminus V)$ has a complement in ${\sf pH}(X)$, finishing the proof that 
   $\pH A$ is Boolean.

    (3)$\Rightarrow$(1): By \cref{prop: M of pH A}, $\BL A \cong \M(\pH A)$. Therefore, $\pH A$ Boolean implies that so is $\BL A$ because the Dedekind-MacNeille completion of a Boolean algebra is a Boolean algebra (see, e.g., \cite[p.~239]{balbes_distributive_1974}).
\end{proof}

As a consequence, we obtain: 

\begin{corollary} \cite[Thm.~3.11]{janowitz_section_1968}
    For $A\in{\DLat}$, $\M A$ is Boolean iff $A$ is $\vee$-subfit and $\wedge$-subfit.
\end{corollary}

\begin{proof}
    First suppose $\M A$ is Boolean. Then $\M A$ is $\vee$-subfit, so $A$ is $\vee$-subfit by \cref{thm: MA subfit implies A subfit}. Moreover, $\M A$ is a frame (since every complete Boolean algebra is a frame; see, e.g., \cite[p.~53]{balbes_distributive_1974}), so $\M A \cong \BL A$ by \cref{thm: MA frame}. Therefore, $\BL A$ is Boolean, so $A$ is $\wedge$-subfit by \cref{thm: BL boolean}. 

    Conversely, suppose $A$ is $\vee$-subfit and $\wedge$-subfit. Since $A$ is $\vee$-subfit, $A$ is proHeyting by \cref{prop: subfit implies proH}, so $\M A \cong \BL A$ by \cref{thm: MA frame}. Because $A$ is $\wedge$-subfit, $\BL A$ is Boolean by \cref{thm: BL boolean}. Thus, $\M A$ is Boolean.
\end{proof}

We conclude this section by characterizing when the ideal and canonical completions are Boolean. As we will see, the condition for the ideal completion to be Boolean is more restrictive.

\begin{proposition} 
    For $A\in{\DLat}$, $\I A$ is Boolean iff $A$ is finite and Boolean.
\end{proposition}

\begin{proof}
If $A$ is finite, then $\I A \cong A$, so $A$ Boolean implies that $\I A$ is Boolean. Conversely, suppose $\I A$ is Boolean. Then every ideal, being complemented, is principal, so again $\I A \cong A$. Since every infinite boolean algebra has an infinite pairwise disjoint set \cite[Prop.~3.4]{koppelberg_s_handbook_1989}, it must also have a nonprincipal ideal. Thus, $A$ is finite. 
\end{proof}

\begin{proposition} \label{prop: Asigma boolean}
    For $A\in{\DLat}$, $A\!^\sigma$ is Boolean iff $A$ is Boolean. 
\end{proposition}

\begin{proof}
Let $X$ be the Priestley space of $A$. Then $A$ is Boolean iff the order on $X$ is discrete (see, e.g., \cite[p.~119]{gratzer_lattice_2011}), which happens iff ${\sf Up}(X)=\mathcal P(X)$, which is equivalent to  $A\!^\sigma$ being Boolean. 
\end{proof}

The situation summarizes as follows: 

\begin{summary}
    Let $A\in\DLat$.
    \begin{enumerate}
    
     \item $\BL A$ (resp.~$\pH A$) is Boolean iff $A$ is $\wedge$-subfit iff $\max X$ is dense in $X$. 
     \item  $\M A$ is Boolean  iff $A$ is both $\vee$- and $\wedge$-subfit iff both $\min X$ and $\max X$ are dense in $X$.
     \item $A\!^\sigma$ is Boolean iff $A$ is Boolean iff $\max X$ is $X$.
     \item $\I A$ is Boolean iff $A$ is finite and Boolean iff $\max X$ is $X$ and $X$ is finite.
        \end{enumerate}
\end{summary}

\section{Conclusion with summary} 

 We conclude by summarizing our findings in \cref{tab:subfit,tab:regular,tab:Boolean}. 
Each table considers one of the three separation properties. The right column describes the property of the completion of a distributive lattice, the left column the corresponding property satisfied by the lattice, and the center column the dual condition on the Priestley space. Within each table, the conditions are arranged from the weakest to the strongest. 

For $A \in \DLat$, recall the following notions from the end of \cref{Sec: subfit} relating to subfitness:

\begin{samepage}
\begin{itemize}

    \item A is \emph{$\BL$-subfit} if $\BL A$ is subfit (equivalently, $\pH A$ is subfit);
    \item A is \emph{$\I$-subfit}\footnote{This should not be confused with the notion of ideally subfit introduced in \cite[Def.~2.13]{delzell_conjunctive_2021} (which is of novel interest only when $A$ is unbounded).} if $\I A$ is subfit;
  
\end{itemize}
\end{samepage}

\begin{table}[ht]
    \centering
    \begin{tabular}{|c|c|c|}
    \hline
    \bf{Lattice}& \bf{Dual space}&\bf{Completion} \\
    \hline\hline
     $\BL$-subfit  & $\min pX$ is dense in $pX$ & $\BL A$ is subfit  \\
        \hline
      subfit  & $\min X$ is dense in $X$ & $\M A$  is subfit \\
      \hline
        $\I$-subfit  & $\min X$ is dense & $\I A$ is subfit \\
        &in the Skula topology & \\
        \hline
         Boolean  & $\min X$ is X & $A\!^\sigma$ is subfit \\
         \hline
         
    \end{tabular}
    \caption{Summary for subfitness}
    \label{tab:subfit}
\end{table}

For $A \in \DLat$, consider the following definitions relating to regularity:
\begin{itemize}
    \item A is \emph{$\M$-regular} if $\M A$ is regular;
    \item A is \emph{$\BL$-regular} if $\BL A$ is regular (equivalently, $\pH A$ is regular);
   
\end{itemize}
For Dedekind-MacNeille completions that are distributive, the notion of $\M$-regular is simply $\BL$-regular plus subfit (\cref{prop: reg MA }). Also, the notions of $\I$- and $(\cdot)^\sigma$-regular coincide with being Boolean (\cref{I-reg is Boolean,can-reg is Boolean}).

\begin{table}[ht]
    \centering
    \begin{tabular}{|c|c|c|}
    \hline
     \bf{Lattice}& \bf{Dual space}&\bf{Completion}\\
    \hline\hline
    
     $\BL$-regular  & For each $U \in {\sf L}(X)$:& $\BL A$ is regular \\
   
    &$R_{\sf BL}(U)$ is dense in $U$ &\\
        \hline
         $\M$-regular  &  $\min X$ is dense in $X$ & $\M A$ is regular \\
    &and for each $U \in {\sf L}(X)$: &\\
    &$R_{\sf BL}(U)$ is dense in $U$ &\\
        \hline
      regular  & for each $U \in {\sf L}(X)$: & $\BL A$  is $A$-regular\\
 &$R(U)$ is dense in $U$  &\\
      \hline
       
         Boolean  & For each $U \in {\sf L}(X)$: & $\I A$ is regular\\
   &$R(U)=U$  &(equiv., $A^\sigma$ is regular)\\
   
        \hline
  
    \end{tabular}
    \caption{Summary for regularity}
    \label{tab:regular}
\end{table}

For $A \in \DLat$, consider the following definitions relating to Booleanness:
\begin{itemize}
    \item A is \emph{$\M$-Boolean} if $\M A$ is Boolean;
    \item A is \emph{$\BL$-Boolean} if $\BL A$ is Boolean (equivalently, $\pH A$ is Boolean);
    \item A is \emph{$\I$-Boolean} if $\I A$ is Boolean;
   
\end{itemize}
The notion of $(\cdot)^\sigma$-Boolean is simply Boolean (\cref{prop: Asigma boolean}).

\begin{table}[ht]
    \centering
    \begin{tabular}{|c|c|c|}
    \hline
    \bf{Lattice}& \bf{Dual space}&\bf{Completion}\\
    \hline\hline
     $\BL$-Boolean  & $\max X$ is dense in $X$ & $\BL A$ is Boolean \\
     
        \hline
      $\M$-Boolean  & $\max X$ and $\min X$ are dense in $X$   & $\M A$  is Boolean\\
       \hline
        Boolean  & $\max X$ is $X$ &$A\!^\sigma$ is Boolean\\
       
        \hline
         $\I$-Boolean  & $\max X$ is $X$ and $X$ is finite  & $\I A$ is Boolean\\

        \hline
       \end{tabular}
    \caption{Summary for Booleanness}
    \label{tab:Boolean}
\end{table}

\noindent {\bf Acknowledgment}: We are thankful to the referee for careful reading and useful comments which have improved the presentation.

\newcommand{\etalchar}[1]{$^{#1}$}

\end{document}